\documentclass[12pt]{amsart}
\usepackage{graphicx}
\usepackage{amsmath}
\usepackage{amsthm}
\usepackage{amsfonts}
\usepackage{amssymb}
\usepackage{color}
\usepackage{verbatim}

\theoremstyle{plain}
\newtheorem{thm}{Theorem}[section]
\newtheorem{lemma}[thm]{Lemma}

\newtheorem{prop}[thm]{Proposition}

\newtheorem*{thm*}{Theorem}
\newtheorem*{prop*}{Proposition}
\newtheorem*{lemma*}{Lemma}

\theoremstyle{remark}
\newtheorem{rmk}[thm]{Remark}
\newtheorem{example}[thm]{Example}

\theoremstyle{definition}
\newtheorem{defn}[thm]{Definition}

\numberwithin{equation}{section}

\newenvironment{PfofTracking}[1]
{\par\vskip2\parsep\noindent{\sc Proof of Theorem\ \ref{Thm:GeneralConvergence}. }}{{\hfill
$\Box$}
\par\vskip2\parsep}

\newenvironment{PfofOptJoinings}[1]
{\par\vskip2\parsep\noindent{\sc Proof of Theorem\ \ref{Thm:BasicPropertiesOfJmin}. }}{{\hfill
$\Box$}
\par\vskip2\parsep}

\newenvironment{PfofInference}[1]
{\par\vskip2\parsep\noindent{\sc Proof of Theorem\ \ref{Thm:Inference}. }}{{\hfill
$\Box$}
\par\vskip2\parsep}

\newenvironment{PfofQuantConsist_I}[1]
{\par\vskip2\parsep\noindent{\sc Proof of Theorem\ \ref{Thm:QuantConsist}, Part (1). }}{{\hfill
$\Box$}
\par\vskip2\parsep}

\newenvironment{PfofQuantConsist_II}[1]
{\par\vskip2\parsep\noindent{\sc Proof of Theorem\ \ref{Thm:QuantConsist}, Part (2). }}{{\hfill
$\Box$}
\par\vskip2\parsep}

\newenvironment{PfofCircleRotations}[1]
{\par\vskip2\parsep\noindent{\sc Proof of Proposition\ \ref{Prop:CircleRotations}. }}{{\hfill
$\Box$}
\par\vskip2\parsep}

\usepackage[colorinlistoftodos,textsize=tiny]{todonotes} 
\newcommand{\Comments}{1}
\newcommand{\mynote}[2]{\ifnum\Comments=1\textcolor{#1}{#2}\fi}

\def\E{\mathbb{E}}

\def\R{\mathbb{R}}
\def\N{\mathbb{N}}

\def\cF{\mathcal{F}}

\def\cJ{\mathcal{J}}

\def\cL{\mathcal{L}}
\def\cM{\mathcal{M}}
\def\cP{\mathcal{P}}

\def\cR{\mathcal{R}}

\def\cU{\mathcal{U}}
\def\cV{\mathcal{V}}
\def\cX{\mathcal{X}}
\def\cY{\mathcal{Y}}

\def\lab{\mbox{lab}}

\def\Tracking{Tracking Theorem}

\DeclareMathOperator{\dbar}{\overline{\textit{d}}}

\DeclareMathOperator{\Id}{Id}

\DeclareMathOperator{\proj}{proj}
\DeclareMathOperator{\supp}{supp}

\DeclareMathOperator{\argmax}{argmax}
\DeclareMathOperator*{\argmin}{argmin}

\newcommand{\pf}[2]{m_{#2}(#1)}

\title{Variational analysis of inference from dynamical systems}

\begin{document}

\author{Kevin McGoff and Andrew B. Nobel}
\address{Kevin McGoff\\
Department of Mathematics\\
University of North Carolina at Charlotte\\
Charlotte, NC 28223}
\email{kmcgoff1@uncc.edu}
\urladdr{https://clas-math.uncc.edu/kevin-mcgoff/}
\address{Andrew B. Nobel\\
Department of Statistics and Operations Research\\
University of North Carolina at Chapel Hill\\
308 Hanes Building\\
Chapel Hill, NC 27599}
\email{nobel@email.unc.edu}

\begin{abstract}
We introduce and study a variational framework for the analysis of empirical risk based
inference for dynamical systems and ergodic processes.  
The analysis applies to a two-stage estimation procedure 
in which (i) the trajectory of an observed (but unknown) system is fit to a 
trajectory from a known reference system by minimizing cumulative 
per-state loss, 
and (ii) a parameter estimate is obtained from the initial state of the best
fit reference trajectory.
We show that the empirical risk of the best fit trajectory  
converges almost surely to a constant that can be expressed
in a variational form
as the minimal expected loss over dynamically invariant couplings (joinings) of the 
observed and reference systems.  Moreover, we establish that 
the family of joinings minimizing the expected loss is convex and compact, and that it
fully characterizes the asymptotic
behavior of the estimated parameters, addressing both identifiability and misspecification.   
The two-stage estimation framework and associated variational analysis 
apply to a broad family of empirical risk miminization
procedures for dependent observations.  To illustrate this, we apply variational
analysis to the well studied problems of maximum likelihood and non-linear regression,
and then undertake an extended analysis of system identification 
from quantized trajectories subject to noise, a problem of interest in dynamics, where
the models themselves exhibit dynamical behavior across time.
\end{abstract}


\maketitle


\section{Introduction} \label{Sect:Tracking_Intro}

Minimization of empirical risk is a common approach to statistical inference, having roots in maximum likelihood estimation 
from independent observations. 
Unlike maximum likelihood estimation, however, the form and motivation of empirical risk minimization remain unchanged when considering stationary, dependent observations.

In this paper we introduce and study a variational framework for the analysis of empirical risk based inference 
from dynamical systems and ergodic processes.  
The variational analysis applies to a family of procedures that can be decomposed into two stages: a {\em tracking} stage in which a trajectory of a known reference system is fit to the trajectory of an observed system by minimizing cumulative loss; and a {\em translation} stage in which a parameter estimate is obtained by applying a continuous invariant map to the initial state of the best-fit reference trajectory.  
We emphasize that tracking plus translation is not being proposed as an inference procedure, 
but rather as a framework within which to study empirical risk minimization.  
In Section \ref{Sect:Applications} we show that the two-stage inference framework encompasses the well studied problems 
of maximum likelihood estimation and non-linear regression, as well as new, more complex problems, 
including fitting dynamical models and system identification from quantized trajectories.

This paper makes three main contributions. 
The first contribution is the introduction of ideas and techniques from ergodic theory to the study of 
empirical risk minimization through an analysis of two-stage inference.
The most important of these ideas is the notion of a joining, which is a dynamically invariant coupling of two systems (see Section \ref{Sect:Tracking} for definition and further discussion).  
Our principal results concerning two-stage inference are the following.
First, the average loss of the optimal reference trajectory in the tracking stage converges almost surely to a 
constant that can be expressed in variational form as the minimum expected loss over all joinings of the observed 
and reference systems.  Second, the family of optimal joinings, namely those that achieve the minimum expected loss
in the variational problem, is non-empty, convex, 
and compact in the weak topology.
Third, the family of optimal joinings characterizes the limiting behavior of the estimates derived in the translation stage.

Together, these results constitute a variational analysis of the two-stage
inference framework.
In particular, they provide an explicit, variational characterization of the asymptotic behavior of two-stage inference 
that establishes a direct connection (not previously known or understood) between minimum risk-based inference 
and an infinite dimensional optimization problem in which the observation process is projected, 
using a joining-based distortion, onto a family of processes that is associated with the statistical models under study.
This variational analysis is the second main contribution of the paper.

Variational analysis has a number of desirable properties.  
It requires relatively mild assumptions (given in detail below).  It readily accommodates model misspecification, 
as the observed and reference systems need not be related to one another.
It addresses the problem of identifiability in a direct way, by fully characterizing 
the limiting parameter set, and it provides a systematic means of studying the limiting behavior and potential consistency 
of two-stage inference.  

The third main contribution of the paper is an extended application of variational analysis to the problem of  
system identification from quantized trajectories with label noise, a problem that lies at the intersection
of statistics and dynamical systems.  In this case, 
analysis requires results and constructions from the theory of joinings, illustrating the full power
of the variational approach.

In the next two sections we describe the two-stage inference procedure and present the principal results underlying 
variational analysis.  
Several applications of variational analysis to several existing and new inference problems are given in Section \ref{Sect:Applications}, and our main application, the analysis of system identification
with quantized trajectories, is given in Section \ref{Sect:QuantizedIntro}.

%
%
%

\section{The tracking stage} \label{Sect:Tracking}

The tracking stage of the two-stage inference framework has three basic components: an observed ergodic system, a reference topological system, 
and an integrable loss function.
  
The observed dynamical system is a triple $(\cY, T, \nu)$ consisting of a non-empty Polish space 
$\cY$, a Borel measurable map $T : \cY \to \cY$, and a Borel probability measure $\nu$ on $\cY$
that is invariant and ergodic under $T$.  
Recall that $\nu$ is invariant under $T$ if the action of $T$ preserves $\nu$ in
the sense that $\nu(T^{-1}A) = \nu(A)$ for each Borel set $A \subseteq \cY$.
Furthermore, recall that $\nu$ is ergodic under $T$ if $T^{-1} A = A$ implies $\nu(A) \in \{ 0, 1 \}$, that is, any
set that is invariant under the action of $T$ is negligible or has full measure.
The triple $(\cY, T, \nu)$ is the standard model of a stationary ergodic dynamical system
\cite{Petersen1983,Walters2000}.  It is well known \cite{Breiman1992,Durrett2010} that any stationary ergodic process can be expressed as such a triple (with $\cY$ equal to a sequence space, $T$ equal to the left shift transformation, and $\nu$ equal to the process distribution) plus a univariate projection.   
     
By contrast with the observed dynamical system, the reference system is a pair $(\cX, S)$ 
consisting of a non-empty, compact metric space $\cX$ 
and a continuous map $S: \cX \to \cX$.  In the dynamics literature, the pair $(\cX, S)$ is referred
to as a topological dynamical system \cite{DGS,Walters2000}.
The loss $\ell: \cX \times \cY \to \R$ is assumed to be a lower semicontinuous function 
satisfying the envelope condition $\sup_{x \in \cX} |\ell(x,y)| \leq \ell^*(y)$ for some $\ell^* \in L^1(\nu)$. 
{\it In what follows, the conditions above will be referred to as the standard assumptions.}

\vskip.1in

{\bf Tracking.} In the tracking problem we have access to a single trajectory $y$, $T y$, $T^2 y$, $\ldots$ of the observed 
system $(\cY, T, \nu)$ with initial state $y$ drawn according to $\nu$.  
Here $T^k$ denotes the k-fold composition of $T$ with itself and, by convention, $T^0$ is the identity.
At time $n$ we observe the initial segment $y, Ty, \ldots, T^{n-1}y$ of the trajectory and 
identify a corresponding initial state $x$ of the reference system $(\cX, S)$ that
minimizes the average loss
$
n^{-1} \sum_{i=0}^{n-1} \ell(S^i x, T^i y) 
$.  
Our assumptions on $\ell$ and $(\cX, S)$ ensure that a minimizing initial condition exists.

\vskip.1in

The reader will note that there is an asymmetry in the specification of the observed and reference systems: 
the observed system is equipped with an invariant ergodic measure, while the reference system is specified without 
reference to any particular invariant measure.  Indeed, we view the reference system as a {\em family} of 
stationary dynamical systems corresponding to the family of $S$-invariant measures.
In more detail, let $\mathcal{M}(\cX,S)$ be the family of all Borel probability 
measures on $\cX$ that are invariant under $S$ in the sense that $\mu \circ S^{-1} = \mu$.
It can be shown \cite{KrylovBogolioubov} that the family $\mathcal{M}(\cX,S)$ is non-empty.
Importantly, every measure $\mu \in \mathcal{M}(\cX,S)$ determines a stationary, but not necessarily
ergodic, dynamical system $(\cX,S,\mu)$.  Let
\[
\mathcal{S}(\cX,S) = \{ (\cX,S,\mu) : \mu \in \mathcal{M}(\cX,S)\}
\]
be the family of all such systems.  As we show below, the asymptotic behavior of two-stage inference
is determined by the relationship between the observed system $(\cY, T, \nu)$ and the family
$\mathcal{S}(\cX,S)$.  Quantifying this relationship requires two concepts.
The first is the notion of joining, which is a dynamically invariant couplings of two stationary systems;
the second is a variationally defined distortion between stationary systems. 

\vskip.1in

\begin{defn}
Let $S \times T: \cX \times \cY \to \cX \times \cY$ be the product transformation of 
$S$ and $T$, defined by $(S \times T)(x,y) = (Sx, Ty)$, and
let $\mu \in \mathcal{M}(\cX,S)$.
A Borel probability measure $\lambda$ on $\cX \times \cY$ is a said to be a {\em joining} of 
the systems $(\cX,S,\mu)$ and $(\cY,T,\nu)$ if $\lambda$ is invariant under $S \times T$ and the marginals 
of $\lambda$ on $\cX$ and $\cY$ are $\mu$ and $\nu$, respectively.  
\end{defn}

\vskip.1in

A joining is a coupling of the measures $\mu$ and $\nu$ with the additional property that the coupling is invariant 
(stationary) under the product transformation $S \times T$. The simple example given below gives an indication of the role that this invariance condition plays. 
Joinings were introduced and first studied by Furstenberg \cite{Furstenberg1967}, 
and they have played an important role in ergodic theory since then, see \cite{Rue2006,Glasner2003} for more details.

For each $\mu \in \mathcal{M}(\cX,S)$ let $\cJ(\mu, \nu)$ denote the family of all joinings 
of $(\cX,S,\mu)$ and $(\cY,T,\nu)$.  Note that $\cJ(\mu, \nu)$ contains 
the product measure $\mu \otimes \nu$ and is therefore non-empty.
Define 
\[
\cJ(S: \nu) = \bigcup_{\mu \in \mathcal{M}(\cX,S)} \, \cJ(\mu, \nu),
\]
to be the set of all joinings of the observed system with some system in $\mathcal{S}(\cX,S)$.
It is easy to see that $\cJ(S: \nu)$ is just the set of Borel probability measures $\lambda$ on 
$\cX \times \cY$ such that $\lambda$ is invariant under $S \times T$ and the $\cY$-marginal 
of $\lambda$ is $\nu$.

\begin{example}
As a simple example of a joining, let $\mathcal{X} = \mathbb{Z} / 2 \mathbb{Z}$ with 
$S(x) = x + 1 \mod 2$, and let $\mathcal{Y} = \mathbb{Z} / 3 \mathbb{Z}$ with $T(y) = y+1 \mod 3$. There is only one measure $\mu$ in $\mathcal{M}(\mathcal{X},S)$, which places measure $1/2$ on both $0$ and $1$. Similarly, there is only one invariant and ergodic measure $\nu$ for $(\mathcal{Y},T)$, which places measure $1/3$ on $0,1$, and $2$.  To form a joining of these systems, we need a probability measure on $\mathcal{X} \times \mathcal{Y}$ that is invariant under $S \times T$. 
As $2$ and $3$ are relatively prime, orbits of the map $S \times T$ cycle among the pairs 
$(0,0)$, $(1,1)$, $(0,2)$, $(1,0)$, $(0,1)$, $(1,2)$ in that order. 
Hence, there is only one invariant measure for $S \times T$, the uniform measure, which is the product joining of $\mu$ and $\nu$. Observe that in this case $\mathcal{J}(S : \nu)$ contains only one element, despite the fact that there are uncountably many couplings of $\mu$ and $\nu$.
\end{example}

Our first principal result shows that the average risk in the tracking problem has a limit, 
and that the limit
has a simple variational form involving the loss function $\ell$ and the family of joinings $\cJ(S: \nu)$.
The proof of the theorem is given in Section \ref{Sect:TrackingLemma}.  

\vskip.1in

\begin{thm}[Variational expression for limiting average risk] 
\label{Thm:GeneralConvergence}
Under the standard assumptions, for $\nu$-almost every initial state $y \in \cY$,
\begin{equation}
\label{Tracking1} 
\lim_n  \inf_{x \in \cX} \frac{1}{n} \sum_{k=0}^{n-1} \ell(S^k x, \, T^ky)
\ = 
\inf_{\lambda \in \cJ(S: \nu)} \int \ell \, d\lambda
\, := \, 
L(S: \nu) .
\end{equation}
The second infimum is finite, and is attained by some joining $\lambda$ in $\cJ(S : \nu)$.
\end{thm}

\begin{rmk}[Distortion and Projections]
Note that $L(S:\nu)$ can be written as $\inf_{\mu \in \mathcal{M}(\cX,S)} d(\mu : \nu)$ where
$d(\mu : \nu) = \inf_{\lambda \in \cJ(\mu: \nu)} \int \ell \, d\lambda$ defines a joining-based 
distortion between the stationary systems $(\cX,S, \mu)$ and $(\cY,T,\nu)$ under the loss function $\ell$.  
Thus the limiting average risk $L(S: \nu)$ in the tracking problem is the minimum distortion between the 
observed ergodic system $(\cY,T,\nu)$ and the family 
$\mathcal{S}(\cX,S)$ associated with the reference system $(\cX,S)$.
It is natural then to investigate the joinings in $\cJ(S:  \nu)$ that achieve the minimum 
$L(S: \nu)$.  These joinings capture the joint behavior of the observed system $(\cY,T,\nu)$ and its 
projection onto the family $\mathcal{S}(\cX,S)$ under the distortion $d(\cdot : \cdot)$.  
\end{rmk}

Recall that $\mu$ is said to be an extreme point of a convex family of probability measures 
$\cM$ if $\mu = t \mu_1 + (1-t) \mu_2$ 
with $t \in (0,1)$ and $\mu_1,\mu_2 \in \cM$ implies $\mu_1 = \mu_2 = \mu$.
The proof of the following theorem can be found in Appendix \ref{Sect:SetOfOptimalJoinings}.

\begin{thm}[Structure of optimal joinings] 
\label{Thm:BasicPropertiesOfJmin}
Under the standard assumptions, the set of optimal joinings
\[
\cJ_{\tt{min}}( S: \nu) \, = \, \Big\{ \lambda \in \cJ( S: \nu) : \int \ell \, d\lambda = L(S: \nu) \Big\}
\] 
is non-empty, convex, and compact in the weak topology.
Furthermore, a joining $\lambda$ is an extreme point of $\cJ_{\tt{min}}( S:  \nu)$ if and only if it 
is ergodic under $S \times T$.
\end{thm}

\begin{rmk} \label{Rmk:ErgJoining}
It follows from Theorem \ref{Thm:BasicPropertiesOfJmin} and the Krein-Milman 
theorem that there exists an optimal joining that is ergodic under $S \times T$.
By considering the $\cX$-marginal of such a joining, we see that there exists an ergodic measure 
$\mu$ for the system $(\cX,S)$ that can be optimally joined with $\nu$. 
\end{rmk}

\section{Translation} \label{Sect:Translation}

For a fixed sample size $n$ 
the tracking problem is a special case of empirical risk minimization
in which the observed sequence $y, Ty, \ldots, T^{n-1}y$ is fit using a family of sequences
$x, Sx,\ldots, S^{n-1}x$ indexed by initial states $x \in \cX$ of the reference system.  In the translation
stage of the inference procedure, a parameter estimate is obtained from the initial state
of a reference trajectory that minimizes, or nearly minimizes, the empirical risk.  

Let $\Theta$ be a compact metrizable parameter space.  A {\em parameter map} is a continuous function 
$\varphi: \cX \to \Theta$ that is invariant under the dynamics of the reference system in the sense
that $\varphi \circ S = \varphi$.  {\it In what follows these conditions on $\Theta$ and $\varphi$ are
included in the standard assumptions.}
The invariance of $\varphi$ ensures that the value $\theta = \varphi(x)$ is constant 
on the trajectory $x, S x, S^2 x, \ldots$, and $\theta$ may therefore be viewed as a property of the entire trajectory of $x$ under $S$.

\begin{defn}[$\varphi$-estimation] \label{Defn:PhiEstimation}
A sequence of measurable functions $\theta_n: \cY^n \to \Theta$, $n \geq 1$, is an 
{\em optimal $\varphi$-estimation scheme} if $\theta_n = \varphi \circ x_n$
where the functions $x_n : \cY^n \to \cX$ are such that for $\nu$-almost every $y$ in $\cY$,
with 
$\hat{x}_{n} = x_n(y, \ldots, T^{n-1} y)$, 
\begin{equation} \label{Eqn:MinimizingScheme}
\lim_n \frac{1}{n} \sum_{k=0}^{n-1} \ell( S^k \hat{x}_{n}, \, T^k y) 
\, = \, 
\lim_n \inf_{x \in \cX} \frac{1}{n} \sum_{k=0}^{n-1} \ell( S^k x, T^k y) .
\end{equation}
Thus the estimate $\hat{\theta}_n = \theta_n (y, \ldots, T^{n-1}y)$ is obtained
by applying the parameter map $\varphi$ to an initial state 
$\hat{x}_n$ of the reference system obtained by minimizing (or nearly minimizing) the
average loss with the observed trajectory
$y, T y, \ldots, T^{n-1}y$.
\end{defn}

\vskip.1in

We wish to characterize the limiting behavior of optimal $\varphi$-estimation schemes $\{ \hat{\theta}_n \}$.
To this end, for each $\theta \in \Theta$ let $\cX_\theta = \varphi^{-1}\{\theta\}$ be the set of states in $\cX$ 
that are mapped to $\theta$, and let $S_{\theta}$ be the restriction of the continuous map $S$ to $\cX_\theta$.  
It is easy to see that $\cX_\theta$ is a compact subset of $\cX$ that is invariant under $S_\theta$, and therefore 
$(\cX_\theta, S_{\theta})$ is a topological dynamical system.  
Thus the parameter map $\varphi$ gives rise to a family of topological systems, indexed by the parameters 
$\theta \in \Theta$, each of which can act as a reference system for tracking the observed system $(\cY, T, \nu)$.

It follows from Theorem \ref{Thm:GeneralConvergence} that the limiting average 
loss of tracking $(\cY, T, \nu)$ with the reference system $(\cX_\theta, S_{\theta})$ 
is equal to $L(S_{\theta}: \nu)$.  
It is not difficult to show (see Lemma \ref{Lemma:ThetaVariational} below) that 
$L(S : \nu) = \min_{\theta \in \Theta} L(S_{\theta} : \nu)$, and we therefore 
consider the set of parameters with minimal limiting loss, namely,
\begin{equation} 
\label{Eqn:DefOfThetaMin}
\Theta_{\tt{min}} = \argmin_{\theta \in \Theta} \, L\bigl(S_{\theta}: \nu \bigr)  . 
\end{equation}
It is straightforward to show that 
\begin{equation} \label{Eqn:EquivDefThetaMin}
\Theta_{\tt{min}} = 
\bigl\{ \theta \in \Theta : \exists \, \lambda \in \cJ_{\tt{min}}(S:  \nu) \mbox{ s.t.\ } \lambda( \cX_{\theta} \times \cY) = 1 \bigr\} 
\end{equation}
is just the set of parameters $\theta$ whose associated states 
$\cX_\theta \subseteq \cX$ support an optimal joining 
with the observed system $(\cY, T, \nu)$. 
In other words, $\Theta_{\tt{min}}$ is the set of parameters that can be optimally joined with the observed system, and it therefore serves as a natural limit set for optimal $\varphi$-estimation schemes.

The next theorem is our principal result concerning two-stage inference. 
Its proof appears in Section \ref{Sect:GeneralResults}. 
Here and throughout the paper, we say that a sequence $\{\theta_n\}_{n \geq 1}$ converges 
to a set $\Theta_0$ if the distance of $\theta_n$ to $\Theta_0$ converges to zero.

\begin{thm}[Convergence of Optimal $\varphi$-Estimators] 
\label{Thm:Inference}
Under the standard assumptions $\Theta_{\tt{min}}$ is non-empty and compact.  
%
Moreover, if $(\theta_n)_{n \geq 1}$ is an optimal $\varphi$-estimation scheme, then 
$\hat{\theta}_n = \theta_n (y, \ldots, T^{n-1}y)$ converges to $\Theta_{\tt{min}}$ for $\nu$-almost every $y \in \cY$.
Conversely, for every $\theta_0 \in \Theta_{\tt{min}}$ there exists an optimal $\varphi$-estimation scheme that 
converges almost surely to $\theta_0$. 
\end{thm}


\vskip.1in

Theorem \ref{Thm:Inference} fully characterizes the limiting behavior of two stage inference.
In particular, it reduces questions about identifiability and consistency 
to the analysis of the set $\Theta_{\tt{min}}$, and the associated family $\cJ_{\tt{min}}(S:  \nu)$ of optimal joinings.  
In this way, the theorem facilitates the application of joining constructions and properties 
to the analysis of empirical risk minimization in statistics and machine learning, see Section \ref{Sect:Applications}
below.  We emphasize that the theorem places no restrictions on the relation 
between the observed and reference systems, which need not be the same.

In summary, the variational analysis presented above provides a starting point for 
the analysis of two-stage empirical risk minimization that may involve dynamical models. 
With these results it is possible to characterize the convergence of two-stage estimators 
in a variety of applications, as illustrated in the following section. 
Next steps would include rates of convergence or distributional convergence theorems.  
Results of this type, while desirable, will involve substantially stronger assumptions and substantial 
problem specific analyses, and are beyond the scope of the present paper.

\section{First examples of variational analysis} \label{Sect:Applications}

The two-stage inference framework may at first appear to be rather restrictive, 
as it involves fitting deterministic observations to deterministic models in the absence of noise,
a problem that is relatively rare in statistical theory and practice.
However, by appropriate choice of the observed system $(\cY,T,\nu)$, the reference system $(\cX,S)$, 
and the loss function $\ell(x,y)$, the framework and the results of Theorem \ref{Thm:Inference}
may be applied to a wide variety of inference problems 
involving empirical risk minimization and dependent observations.  In Sections \ref{Sect:MLE_intro}
and \ref{Sect:RegressionIntro} below, we show how maximum likelihood estimation and nonlinear regression 
can be placed within the framework of two-stage inference, and how Theorem \ref{Thm:Inference} can be used to
establish existing results on the consistency of empirical risk based procedures for these problems.

While these initial applications are illuminating, the two-stage inference framework and Theorem \ref{Thm:Inference} 
also provide powerful tools for analyzing new, more complex, statistical problems in which the observations
as well as the models of interest are dynamical in nature.   
In Section \ref{Sect:Fitting} we briefly review recent work of McGoff and Nobel \cite{McGoffNobel_applications}
on fitting families of dynamical models, which is based in part on the results of Theorem \ref{Thm:Inference}.
The full generality of the two-stage framework and the power of the variational characterization is illustrated 
in Section \ref{Sect:QuantizedIntro}, where we describe and analyze the problem of system identification from quantized trajectories.  
These results are of independent interest and, to the best of our knowledge, cannot be obtained by any existing statistical methods. 

\subsection{Maximum likelihood estimation under ergodic sampling} 
\label{Sect:MLE_intro}

Let $\cU$ be a Polish space, and let $\mathcal{P} = \{ f_{\theta} : \theta \in \Theta \}$ be a family of 
density functions $f_{\theta} : \cU \to [0,\infty)$ with respect to a fixed Borel measure $Q$ on $\cU$.
Let $\Theta$ be a compact metric space and assume that $(\theta, u) \mapsto f_{\theta}(u)$ 
is an upper semi-continuous map from $\Theta \times \cU$ to $\R$. 
Suppose that we observe the values of a stationary ergodic process $U_0, U_1 \ldots \in \cU$ and
wish to identify a density $f_{\theta} \in \cP$ that best approximates the one-dimensional marginal distribution of 
the observed process in the sense that
\begin{equation}
\label{opt-theta}
\E \, \log f_\theta (U) \ = \ \max_{\theta' \in \Theta}  \, \E \, \log f_{\theta'} (U) ,
\end{equation}
where $U$ has the same distribution as $U_0$.
Let $\theta_n: \cU^n \to \Theta$, $n \geq 1$, be measurable estimators that 
asypototically maximize the marginal log-likelihood 
\begin{equation} \label{Eqn:MLE_EstimatorIntro}
\lim_n \frac{1}{n} \sum_{i=0}^{n-1} \log f_{\hat{\theta}_n} (U_i)
\ = \ 
\lim_n \sup_{\theta \in \Theta} \frac{1}{n} \sum_{i=0}^{n-1} \log f_\theta (U_i)
\ \mbox{ wp1},
\end{equation}
where $\hat{\theta}_n = \theta_n(U_0,\ldots U_{n-1})$. 
The existence of measurable estimators satisfying (\ref{Eqn:MLE_EstimatorIntro}) 
follows from standard arguments, see Lemma \ref{Lemma:MeasurabilityOptimizers}.  
Note that the common distribution of the observations $U_i$ need not have a density in
$\mathcal{P}$ and need not be absolutely continuous with respect to the reference measure $Q$.

\vskip.1in

The problem described above can be expressed as a
two-stage inference procedure in the following way.
To begin, we represent the observed process $\{U_i\}_{i \geq 0}$ as a 
shift-based system $(\mathcal{Y},T,\nu)$ following the standard construction 
\cite{Breiman1992,Durrett2010} in which
$\mathcal{Y}$ be the sequence space $\mathcal{U}^{\N}$, $T$ is the left-shift on $\cY$, 
and $\nu$ is the measure on $\mathcal{U}^{\N}$ induced by $\{U_i\}$. 
Let the state space $\cX$ of the reference system be equal to the parameter space $\Theta$ and,
as the inference task involves no dynamics beyond those of the observations $U_i$, let 
$S$ be the identity map on $\Theta$.  Finally, let the loss function $\ell: \cX \times \cY \to \R$ be defined by
$\ell(\theta, (u_i)_{i \geq 0}) = - \log f_\theta (u_0)$, and let $\varphi: \cX \to \Theta$ be the identity map.
These correspondences are detailed in Table \ref{Table:MLEChoices}. 
A direct application of Theorem \ref{Thm:Inference} yields the following classical result,
which is similar to Theorem 5.14 of \cite{van2000asymptotic}.

\begin{center}
\begin{table}[h] 
\begin{tabular}{| c | c | }
\hline   
General setting & MLE under ergodic sampling \\ 
\hline \hline $\cX$ & $\Theta$ \\ 
\hline $S : \cX \to \cX$ & $\Id : \Theta \to \Theta$ \\
\hline $\cY$ & $\cU^{\N}$ \\
\hline $T : \cY \to \cY$ & Left shift $\tau$ on $\cU^{\N}$ \\
\hline $\nu$ & Measure of process $\{U_i\}_{i \geq 0}$ \\
\hline $c: \cX \times \cY \to \R$ & $(\theta, \mathbf{u}) \mapsto - \log p_\theta(u_0)$ \\
\hline $\Theta$ & $\Theta$ \\
\hline $\varphi : \cX \to \Theta$ & $\Id : \Theta \to \Theta$ \\
\hline 
\end{tabular}
\vspace{2mm}
\caption{Correspondence between objects in the general setting and objects in MLE under ergodic sampling.} \label{Table:MLEChoices}
\end{table}
\end{center}

\begin{thm} \label{Thm:MLEConsistency}
If $\mathbb{E} \sup_{\theta \in \Theta} | \log f_\theta(U) |$
is finite, then $\hat{\theta}_n$ 
converges almost surely to the non-empty set 
$\Theta_0 = \argmax_{\theta \in \Theta} \, \mathbb{E} \, \log f_\theta ( U )$.
\end{thm}

Thus even in the misspecified setting, the empirical maximum likelihood 
estimators converge to the set of optimal parameters, namely those that best approximate the 
observed process in the sense of (\ref{opt-theta}).  If the supremum in the theorem fails to 
be measurable, one may replace the expectation there by an outer expectation.

\subsection{Nonlinear regression under ergodic sampling} 
\label{Sect:RegressionIntro}

Let $\cU$ be a Polish space, and let $\cF = \{ f_{\theta} : \theta \in \Theta \}$ be a 
family of functions $f_\theta: \cU \to \R$ indexed by a compact metric space 
$\Theta$ in such a way that $(\theta, u) \mapsto f_{\theta}(u)$ is a continuous 
map from $\Theta \times \cU$ to $\R$.  
Suppose that we observe a stationary ergodic process $(U_0, V_0), (U_1, V_1), \ldots \in \cU \times \R$  
and wish to identify a function $f_{\theta} \in \cF$ 
that best captures the marginal relationship between $U$ and $V$ in the sense that
\begin{equation}
\label{opt-theta}
\E \, \ell_0(f_{\theta}(U),V) \ = \ \min_{\theta' \in \Theta}  \, \E \, \ell_0(f_{\theta'}(U),V),
\end{equation}
where 
$\ell_0 : \R \times \R \to [0,\infty)$ is a lower semicontinuous loss function.  
Let $\theta_n: (\cU \times \R)^n \to \Theta$, $n \geq 1$, be 
measurable estimators that asymptotically minimize the average loss, namely
\begin{equation} \label{Eqn:RegressionEstimatorIntro}
\lim_n \frac{1}{n} \sum_{i=0}^{n-1} \ell_0(f_{\hat{\theta}_n}(U_i), V_i)
\ = \ 
\lim_n \inf_{\theta \in \Theta} \frac{1}{n} \sum_{i=0}^{n-1} \ell_0(f_{\theta}(U_i), V_i)
\ \mbox{ wp1},
\end{equation}
where $\hat{\theta}_n = \hat{\theta}_n((U_0,V_0),\ldots (U_{n-1}, V_{n-1}))$. 
This problem can readily be expressed as a
two-stage inference procedure;
Table \ref{Table:RegressionChoices} gives the details.  The following result is an easy consequence of 
Theorem \ref{Thm:Inference}.  The result can also be established by arguments based on uniform 
laws of large numbers.

\vskip.1in

\begin{thm} \label{Thm:RegressionConsistency}
If $\E \sup_{\theta \in \Theta} \ell_0(f_{\theta}(U), V)$ is finite
then $\hat{\theta}_n$ converges almost surely to the set 
$
\Theta_0 = \argmin_{\theta \in \Theta} \, \mathbb{E} \, \ell_0( f_{\theta}(U),V)
$.
\end{thm}

%

\begin{center}
\begin{table}[h] 
\begin{tabular}{| c | c | }
\hline   
General setting & Regression under ergodic sampling \\ 
\hline \hline $\cX$ & $\Theta$ \\ 
\hline $S : \cX \to \cX$ & $\Id : \Theta \to \Theta$ \\
\hline $\cY$ & $(\cU \times \R)^{\N}$ \\
\hline $T : \cY \to \cY$ & Left shift $\tau$ on $(\cU \times \R)^{\N}$ \\
\hline $\nu$ & Measure of process $(U_i,V_i)_{i \geq 0}$ \\
\hline $\ell: \cX \times \cY \to \R$ & $(\theta, (\mathbf{u},\mathbf{v})) \mapsto \ell_0(f_{\theta}(u_0),v_0)$ \\
\hline $\Theta$ & $\Theta$ \\
\hline $\varphi : \cX \to \Theta$ & $\Id : \Theta \to \Theta$ \\
\hline 
\end{tabular}
\vspace{2mm}
\caption{Correspondence between objects in the general setting and objects in the regression under ergodic sampling section.} \label{Table:RegressionChoices}
\end{table}
\end{center}

\subsection{Fitting dynamical models} \label{Sect:Fitting}

Here we provide a brief overview of some recent results on fitting dynamical models to ergodic processes \cite{McGoffNobel_applications}. These results use the variational analysis of the present work as a starting point for a detailed analysis of a specific inference problem. In contrast to the two previous examples (maximum likelihood estimation and nonlinear regression), the underlying model family in this application is complex and captures interesting dynamical behavior.

A dynamical model consists of a continuous transformation $f : K \to K$ on a compact metric space $K$ and a continuous observation function 
$g : K \to \mathbb{R}^d$. The specific inference problem considered in \cite{McGoffNobel_applications} involves fitting a family of dynamical models to observations from a stationary ergodic process. Let $\Theta$ be a compact metric space, and let 
$\mathcal{D} = \{ (f_{\theta},g_{\theta}) : \theta \in \Theta \}$ be an indexed family of dynamical models on a common state space $K$ such that both $(\theta,x) \mapsto f_{\theta}(x)$ and $(\theta,x) \mapsto g_{\theta}(x)$ are continuous. Examples of natural families of dynamical models include the family of toral rotations and the logistic family (see \cite{Katok}). 
The family $\mathcal{D}$ is meant to capture regularities (deterministic patterns) of interest, and the goal of the inference 
is to identify these regularities or patterns in an observed ergodic process $\mathbf{U}= U_0,U_1,\dots \in \R^d$ 
by fitting the values of the process with models in $\mathcal{D}$.  This type of fitting occurs frequently in fields such as systems biology and ecology, see for example \cite{Brackley2010,Letham2016,Levin2009,McGoff2016_LEM,Turchin2013}. 

The main results of \cite{McGoffNobel_applications} concern the asymptotic behavior of parameter estimates that arise 
from fitting dynamical models in $\mathcal{D}$ to $\mathbf{U}$.  In particular, the results establish that if the model 
class $\mathcal{D}$ has low complexity, then empirical risk based fitting procedures are asymptotically consistent in 
signal plus noise settings.
Theoretical analysis of model fitting begins with the variational analysis of Theorem \ref{Thm:Inference}
and makes use of the theory and construction of joinings.

\section{System identification from quantized trajectories} 
\label{Sect:QuantizedIntro}



As a final application of the variational analysis techniques in Section \ref{Sect:Tracking_Intro}, 
we consider the problem of system identification from quantized trajectories.
As in the fitting of dynamical models, the model family in this problem is complex and 
captures dynamical behavior.

Discretizing the state space of a dynamical system is common in both theoretical and applied settings. 
On the theoretical side, studying systems through their quantized trajectories has been a core idea in ergodic theory 
since the theory of entropy was developed by Kolmogorov, Ornstein, and others (see the books \cite{Petersen1983,Walters2000}). 
Moreover, discretization of trajectories occurs frequently in applied settings as a way of coarse graining the state space, 
accommodating computation, and meeting data transmission and storage requirements. For example, many mathematical models of gene regulatory network dynamics involve discretizing the gene expression levels \cite{Karlebach2008}. Also, many numerical methods for analyzing dynamical systems begin by discretizing the state space; examples include methods based on Conley index theory \cite{Mischaikow1999,Mischaikow2002}, Ulam's method \cite{Bose2001}, or finite element method \cite{Dhatt2012}. 
To our knowledge, the results of this section represent the first detailed and rigorous statistical analysis of inference for
dynamical systems based on discretized observations.  
Our analysis makes use of relatively independent joinings, as well as the 
classical result of Furstenberg concerning the disjointness of zero entropy and Bernoulli systems (described in detail below).  

Let $\cU$ be a Polish space and let $\cR = \{ R_{\theta} : \theta \in \Theta \}$ be 
a family of Borel measurable transformations $R_{\theta} : \cU \to \cU$ indexed by a 
compact metric space $\Theta$ in such a way that the map $(\theta,u) \mapsto R_{\theta}(u)$ 
is a Borel measurable function from $\Theta \times \cU$ to $\cU$.
Let $\{A_0, A_1\}$ be a known, measurable partition of $\cU$, and let 
$\pi: \cU \to \{ 0,1 \}$ be the associated label function defined by $\pi(u) = j$ if $u \in A_j$.   
For each parameter $\theta \in \Theta$ and element $u \in \cU$ there is an associated trajectory 
$u, R_\theta \, u, R_\theta^2 \, u, \ldots$ arising from repeated application of $R_\theta$.  
Application of the map $\pi$ gives rise to a corresponding binary label sequence 
\[
\lab(\theta,u) \ = \ (\pi(u), \pi(R_\theta u), \pi(R_\theta^2 u), \ldots) \, \in \, \{0,1\}^\N .
\]
Of interest here is whether, and in what sense, we can estimate the
parameter $\theta$ from noisy observations of the label sequence $\lab(\theta,u)$
when the state $u$ is drawn from an invariant ergodic measure for $R_\theta$.  

In more detail, we assume that observations take the form of a binary stochastic process
\begin{equation}
\label{eqn:qom}
Y_k \ = \ \pi(R_{\theta_0}^k U) \oplus \varepsilon_k, \ \ \ \ k \geq 0.
\end{equation}
Here $\theta_0 \in \Theta$ is the parameter of the underlying transformation, 
$U$ is a $\cU$-valued random variable whose 
distribution is invariant and ergodic under $R_{\theta_0}$, $\{ \varepsilon_k \}$ is a sequence of
independent Bernoulli$(p)$ random variables independent of $U$, and
$\oplus$ denotes addition modulo 2.
Thus $Y_k$ is equal to the label of $R_{\theta_0}^k U$ perturbed by noise.
Standard arguments ensure that $\{ Y_k : k \geq 0 \}$ is stationary and ergodic. 
Both the setting and the results of this section may be generalized to allow 
arbitrary finite measurable partitions and more general noise distributions (see \cite{McGoffNobel_arxiv}); 
we consider the case of binary observations for clarity of presentation.

Let $\theta_n: \{0,1\}^n \to \Theta$, $n \geq 1$, be measurable estimators of $\theta_0$
that approximately minimize average Hamming (0-1) risk, 
\begin{eqnarray} 
\label{Eqn:DiscreteEstimator}
\lefteqn{ 
\lim_n \, \inf_{u \in \cU} \, 
\frac{1}{n} \sum_{k=0}^{n-1} \mathbb{I}\bigl( \pi \bigl(R_{\hat{\theta}_n}^k u \bigr) \neq Y_k \bigr) 
} \\
& = &
\lim_n \, \inf_{\theta \in \Theta} \inf_{u \in \cU} \, 
\frac{1}{n} \sum_{k=0}^{n-1} \mathbb{I}\bigl( \pi\bigl(R_{\theta}^k u\bigr) \neq Y_k \bigr) 
\ \mbox{ wp1},
\nonumber
\end{eqnarray}
where $\hat{\theta}_n = \theta_n(Y_0,\dots,Y_{n-1})$.
The existence of measurable estimators satisfying (\ref{Eqn:DiscreteEstimator}) 
is guaranteed by Lemma \ref{Lemma:MeasurabilityOptimizers}.  We are interested
in the limiting behavior of $\hat{\theta}_n$.

The estimation problem above can be expressed as a two stage inference procedure.
As in the previous examples, the observation process $\{ Y_k \}_{k \geq 0}$ can be represented as a measure 
preserving system $(\mathcal{Y},T,\nu)$, where $\mathcal{Y}$ is the sequence space $\{ 0,1 \}^{\N}$, 
$T$ is the left-shift on $\{ 0,1 \}^{\N}$, and $\nu$ is the measure of the process $\{ Y_k \}_{k \geq 0}$.  
Specification of the reference system and loss require more care.
Let the state space of the reference system be the closure of all (parameter, label-sequence) pairs,
\[
\cX = 
\mbox{cl} \bigl\{ (\theta, \lab(\theta,u)) : \theta \in \Theta, u \in \mathcal{U} \bigr\} \, \subseteq \, \Theta \times \{0,1\}^{\N},
\]
where $\mbox{cl} A$ denotes the closure of $A$ and we assume that $\{0,1\}^{\N}$ is equipped with the usual
product topology.  Thus $\cX$ is compact, and we let the reference transformation $S$ be the restriction to $\cX$ of
the product $\mbox{id}_\Theta \times \tau$, where $\mbox{id}_\Theta$ is the identity on $\Theta$ and 
$\tau$ is the left shift on $\{0,1\}^{\N}$.  It is easy to see that $\cX$ is invariant under $S$.
We further define the loss $\ell((\theta, \mathbf{a}), \mathbf{b}) = \mathbb{I}( a_0 \neq b_0)$, and we let
the parameter map $\varphi$ be the projection onto $\Theta$, namely $\varphi (\theta, \mathbf{a}) = \theta$.
These correspondences are summarized in Table \ref{Table:DiscreteChoices}. 

Theorem \ref{Thm:Inference} characterizes the limiting behavior of the parameter 
estimates $\hat{\theta}_n$. 
Let $\nu_0$ be the distribution of the true label process $\{ \pi(R_{\theta_0}^k U) : k \geq 0 \}$ on 
$\{0,1\}^\N$.  Note that $\nu_0$ is not equal to the process measure $\nu$ of $\{Y_k\}$ if the 
noise level $p > 0$.  
For each $\theta \in \Theta$ let $\cX_\theta = \{ \mathbf{a} : (\theta, \mathbf{a}) \in \cX \}$ be 
the $\theta$-section of $\cX$, and define
\begin{equation} \label{Eqn:Puma}
\Theta_1 = \bigl\{ \theta \in \Theta : \nu_0(\cX_{\theta}) = 1 \bigr\} 
\end{equation}
to be the set of parameters for which $\cX_\theta$ supports the true label 
process. It is easy to see that $\theta_0 \in \Theta_1$, so $\Theta_1$ is non-empty. 

An equivalent way of viewing the parameter estimator $\theta_n$ is that it picks a 
$\theta$ such that some sequence in $\mathcal{X}_{\theta}$ provides the best Hamming match 
to the observations $Y_0,\dots,Y_{n-1}$. In other words, if $\{ \theta_n \}_{n \geq 1}$ 
satisfies (\ref{Eqn:DiscreteEstimator}), then
\[
\lim_n \, \inf_{\mathbf{a} \in \cX_{\hat{\theta}_n}} \, 
\frac{1}{n} \sum_{k=0}^{n-1} \mathbb{I}\bigl( a_k \neq Y_k \bigr)  \\
 = 
\lim_n \, \inf_{\theta \in \Theta} \inf_{\mathbf{a} \in \cX_{\theta}} \, 
\frac{1}{n} \sum_{k=0}^{n-1} \mathbb{I}\bigl( a_k \neq Y_k \bigr) \ \mbox{ wp1}
\]
where $\hat{\theta}_n = \theta_n(Y_0,\dots,Y_{n-1})$.
Thus $\Theta_1$ is a natural identifiability class in the absence 
of noise ($p=0$). (See Remark \ref{Rmk:Identifiability} below for more discussion of identifiability.)  

In fact, $\Theta_1$ is also an identifiability class when noise is present ($p > 0$), provided there are 
complexity constraints on the family of transformations $\cR$.  To quantify these constraints,
let $\cL$ be the closure (in $\{0,1\}^{\N}$) of the set of all label sequences 
$\{ \lab(\theta,u) : \theta \in \Theta, \, u \in \cU \}$ of the transformations in $\cR$.  Let
\begin{equation} 
\label{Eqn:Htop}
h(\cR) \ = \ \lim_n \frac{1}{n} \log \# \bigl\{ a_0^{n-1} \in \{0,1\}^n : \mathbf{a} \in \cL \bigr\},
\end{equation}
be the exponential growth rate of the number of distinct labeled trajectories of length $n$.  
A more detailed discussion of $\Theta_1$ and $h(\cR)$ can be found in Section \ref{Sect:DiscreteObservations}, which
also contains the proof of the following theorem.

\begin{center}
\begin{table}[h]  
\begin{tabular}{| c | c | }
\hline   
General setting & Quantized observations \\ 
\hline \hline $\cX$ & $\mbox{cl} \{ (\theta, \lab(\theta,u)) : \theta \in \Theta, u \in \mathcal{U} \}$ \\ 
\hline $S : \cX \to \cX$ & $\mbox{id}_\Theta \times \tau$ restricted to $\cX$ \\
\hline $\cY$ & $\{0,1\}^{\N}$ \\
\hline $T : \cY \to \cY$ & Left shift $\tau$ on $\{0,1\}^{\N}$ \\
\hline $\nu$ &  Measure of process $\{Y_k\}_{k \geq 0}$ \\
\hline $\ell: \cX \times \cY \to \R$ & $((\theta, \mathbf{a}), \mathbf{b}) \mapsto \mathbb{I}( a_0 \neq b_0)$ \\
\hline $\varphi : \cX \to \Theta$ & $(\theta,\mathbf{a}) \mapsto \theta$ \\
\hline
\end{tabular}
\vspace{2mm}
\caption{Correspondence between objects in the general setting and objects in the estimation of a transformation with quantized observations.} \label{Table:DiscreteChoices}
\end{table}
\end{center}
%

%
%
%

\vskip.1in

\begin{thm} \label{Thm:QuantConsist}
Let $\{ \theta_n \}$ be a sequence of estimators satisfying 
(\ref{Eqn:DiscreteEstimator}).  If either
\begin{enumerate}

\vskip.1in

\item $p=0$ or

\vskip.1in

\item $0 < p < 1/2$ and $h(\cR) = 0$,

\vskip.12in

\end{enumerate}
then $\hat{\theta}_n$ converges almost surely to $\Theta_1$.  Conversely, for every 
$\theta_0 \in \Theta_1$ there exists a sequence of estimators $\{ \theta_n \}$ satisfying
(\ref{Eqn:DiscreteEstimator}) that converges almost surely to $\theta_0$. 

\end{thm}

\vskip.1in

Note that the limit set $\Theta_1$ is the same for both the noisy and noise-free settings. 
Similar results hold when trajectories are quantized by arbitrary finite partitions and subject to more general noise,  
see \cite{McGoffNobel_arxiv}. 
Note also that Theorem \ref{Thm:QuantConsist} holds without continuity assumptions on the transformations 
$R_{\theta}$ or their indexing by $\theta$. 
However, the topology of $\Theta$ does play an important role in this result, as it affects the closure operation 
that defines $\cX$, which in turn is used to define $\Theta_1$.

\begin{rmk} \label{Rmk:Identifiability}
The identifiability class $\Theta_1$ may be significantly larger than the true parameter $\theta_0$. 
For example, if $\theta_1$ is a parameter such that $\cX_{\theta_1} = \{0,1\}^{\N}$, then $\cX_{\theta_1}$ 
supports all invariant probability measures, in which case $\theta_1$ is certainly contained in $\Theta_1$, 
regardless of the true value of $\theta_0$.  Moreover, note that $\cX_{\theta}$ is defined \textit{after} 
taking the closure to form $\cX$.  Thus, it may happen that $\cX_{\theta}$ contains binary sequences 
that do not appear as the label sequence of {\em any} trajectory from $R_{\theta}$, but are instead limits
of such trajectories. This closure may also cause $\Theta_1$ to contain parameters other than $\theta_0$. (See Proposition \ref{Prop:Identifiability} for a sufficient condition that guarantees that $\Theta_1 = \{\theta_0\}$.) Theorem \ref{Thm:QuantConsist} nonetheless ensures convergence of empirical risk based
estimates to the identifiability class $\Theta_1$, and the converse part of the theorem shows that no smaller set would satisfy the same conclusion.  Identification of $\Theta_1$ can be carried out within 
the variational framework, as the next example shows, but additional analysis is required. 
\end{rmk}

\subsubsection{Circle Rotations}
\label{Ex:CircleRotations}

As a non-trivial application of the results in this section, consider the family 
$\cR$ of circle rotations $R_\alpha: [0,1) \to [0,1)$ defined by 
$R_{\alpha}(x) = x + \alpha \mbox{ mod } 1$, with $\alpha \in \Theta = [0,1/2]$. 
If $\alpha$ is rational with reduced form $m/n$, then $R_{\alpha}^n = \Id$, 
each orbit contains exactly $n$ distinct points, and each ergodic measure is supported on a single orbit. 
On the other hand, if $\alpha$ is irrational, then it is known that Lebesgue measure is the only ergodic Borel 
probability measure for $R_{\alpha}$. 
Consider the partition of $[0,1)$ into sets $A_0 = [0,\frac{1}{2})$ and $A_1 = [\frac{1}{2},1)$. 
The proof of the following result appears in Section \ref{Sect:DiscreteObservations}.



\begin{prop} 
\label{Prop:CircleRotations}
Let $Y_i = \pi(R_{\alpha_0}^i U) \oplus \epsilon_i$, where $\alpha_0 \in [0,1/2]$, the distribution of $U$ 
is invariant and ergodic under $R_{\alpha_0}$, and $\{\epsilon_i\}_{i \geq 0}$ is an independent 
sequence of Bernoulli$(p)$ random variables that is independent of $U$.
If $p < 1/2$ and the estimates $\hat{\theta}_n$ satisfy 
(\ref{Eqn:DiscreteEstimator}), then $\hat{\theta}_n$ converges
almost surely to $\alpha_0$.
\end{prop}



\section{Related work}

A number of recent papers have considered prediction and inference from dynamical systems that evolve deterministically over time, e.g.,  \cite{Hang2016,Hang2016_2,Kutoyants2013,Lalley1999,LalleyNobel2006,McGoff2015,Steinwart2009}. 
While the variational framework considered here focuses on the problem of parameter estimation, the recent work cited above has focused on different aspects of statistical inference. For example, Hang and Steinwart and Steinwart and Anghel  \cite{Hang2016,Hang2016_2,Steinwart2009} give some results about forecasting dynamical systems with specified mixing rates, and Lalley and Nobel \cite{Lalley1999,LalleyNobel2006} establish both positive and negative results for filtering problems in the context of certain dynamical systems. 
For additional references and discussion, see the recent survey on statistical inference  for dynamical systems \cite{McGoffSurvey2015}.

In several applied fields, there is interest in fitting parametrized families of dynamical systems to observations. Indeed, there are examples in ecology \cite{Levin2009,Turchin2013}, geophysical modeling \cite{Bennett2005,Kalnay2003}, systems biology \cite{Brackley2010,Letham2016,McGoff2016_LEM}, and data assimilation \cite{Law2015}. As explained in greater detail in \cite{McGoffNobel_applications}, the variational approach taken here may be useful in analyzing the fitting methods in settings such as these.

Although independence assumptions are common in the statistics and machine learning literature,
there has been long-standing interest, both theoretical and applied, in the analysis of observations
exhibiting long-range dependence.  Representative recent work can be found in
\cite{Adams2010,Alquier2012,Hang2016,Kuznetsov2016,Morvai2005_2,Zimin2016}.
Most work in this area considers rates of convergence and finite sample bounds. 
As noted at the end of Section \ref{Sect:Translation}, such results in the general setting considered here would require substantial additional analysis.

In ergodic theory, Ornstein and Weiss \cite{OrnsteinWeiss1990} considered finitary estimation of a stationary ergodic process from samples of the process. They proposed a specific estimation scheme 
and characterized when the scheme produces consistent estimates of the observed process.
Interestingly, consistent estimation may be possible for restricted classes of systems or processes, as we show in some of our results, despite the fact that consistency is impossible for larger classes of processes (as shown by Ornstein and Weiss). Other related work in ergodic theory concerns finitary estimation of $k$-dimensional distributions for growing $k$ \cite{Marton1994} and finitary estimation of isomorphism invariants \cite{Gutman2008,OrnsteinWeiss2007}.

The minimal expected loss $L(S: \nu)$ and the set of optimal joinings $\cJ_{min}(S: \nu)$ defined here have close analogies in the study of optimal transport; see the book by Villani \cite{Villani2008}.  
Indeed, one of the main goals in optimal transport is to describe the properties of optimal couplings, that is, couplings that achieve the infimum of the expected cost.
Such optimal couplings are analogous to the optimal joinings in $\cJ_{min}(S: \nu)$.
In some cases, notably in the case of Ornstein's $\dbar$-metric and its generalizations in ergodic theory and information theory (see the work of Gray, Neuhoff and Shields \cite{Gray1975} and the book of Gray \cite{Gray2011}), the measures $\mu$ and $\nu$ are taken to be process measures, and the couplings are required to be joinings.
While similar in spirit, our results are distinct from this previous work, since we consider the family of joinings between a topological dynamical system and a measure-preserving system and we focus on applications to inference.

In the special case that the loss function does not depend on the observed trajectory, the tracking part of our two-stage procedure reduces to the problem of ergodic optimization, which has received considerable attention in the mathematical literature in recent years (see the survey of Jenkinson \cite{Jenkinson2006} for a thorough introduction to the topic). 
For some recent results, see the work of Quas and Seifken \cite{Quas2012} and references therein.  

\subsection{Organization of the rest of the paper}

The next section provides some background notation and preliminary lemmas needed for the proofs of the main results.  
Theorems \ref{Thm:GeneralConvergence} 
and \ref{Thm:Inference} are established in Sections \ref{Sect:TrackingLemma} 
and \ref{Sect:GeneralResults}, respectively.
Section 
\ref{Sect:DiscreteObservations} contains
the proofs of consistency for the quantized observation problem presented above.
Appendix \ref{Sect:SetOfOptimalJoinings} contains material on the set of optimal joinings, including the proof of Theorem \ref{Thm:BasicPropertiesOfJmin}.

\section{Definitions and background}
\label{Sect:DefBack}

In this section we provide several preliminary definitions and facts required for the principal results of the paper. 

\subsection{Dynamical systems and spaces of measures}

All topological spaces considered in this paper are Polish (separable and completely metrizable). 
We endow any such space with its Borel $\sigma$-algebra and suppress this choice in our notation. 
Let $\cU$ be a Polish space.  
Following standard notation, $\cM(\cU)$ will denote the space of Borel probability measures on $\cU$ endowed with the usual weak topology, under which $\cM(\cU)$ is itself a Polish 
space.  
Recall that if $R : \cU \to \cU$ is a measurable transformation, then $\cM(\cU,R)$ denotes the set of measures $\mu \in \cM(\cU)$ that are invariant under $R$.
If $h : \cU \to \cV$ is measurable, then we define the ``push-forward'' map $\pf{\cdot}{h} : \cM(\cU) \to \cM(\cV)$ by $\pf{\eta}{h} = \eta \circ h^{-1}$.
If $\cU$ is a non-empty compact metric space and $R$ is continuous, then we refer to $(\cU,R)$ as a topological dynamical system.
It is well-known that in this case, $\cM(\cU,R)$ is non-empty and compact in the weak topology (see \cite{Walters2000}).

\subsection{Product spaces and the shift map}

The canonical projections of a product space $\cU \times \cV$ onto its constituent sets will be denoted by $\proj_{\cU}$ and $\proj_{\cV}$.
If $\lambda$ is a measure on $\cU \times \cV$, then its marginal distributions on $\cU$ and $\cV$ will be denoted by $\pf{\lambda}{\cU}$ and $\pf{\lambda}{\cV}$, respectively.
In several places throughout the paper we will consider infinite product spaces of the form $\cU^{\N}$, where $\cU$ is a Polish space.  
In each case $\cU^{\N}$ is endowed with its product topology and associated Borel sigma-field; elements of $\cU^{\N}$ are denoted as sequences $\mathbf{u} = (u_i)_{i \geq 0}$.
For any product space $\cU^{\N}$ the left-shift map $\tau : \cU^{\N} \to \cU^{\N}$ is defined 
by $\tau(u_0, u_1,\ldots) = (u_1, u_2, \ldots)$. 
Note that $\tau$ is continuous in the product topology. 
For any sequence $(u_l)_{l \geq 0}$ and $0 \leq i \leq j$, we define $u_i^j = (u_i, \ldots, u_j)$.

\subsection{The process generated by a measure-preserving system} \label{Sect:Process}

Let the dynamical systems $(\cX, S)$ and $(\cY, T, \nu)$ satisfy the standard assumptions (stated in Section \ref{Sect:Tracking}).  
By definition of the left shift $\tau$, any probability measure $\tilde{\nu} \in \cM(\cY^{\N}, \tau)$ is the distribution of a one-sided stationary process with values in $\cY$.
We will say that a measure $\tilde{\nu} \in \cM(\cY^{\N}, \tau)$ is {\it generated} by the system $(\cY, T, \nu)$ if the one-dimensional marginal distribution of $\tilde{\nu}$ is $\nu$, and if $\tilde{\nu}$ is supported on trajectories of $T$ in the sense that 
\begin{equation} 
\label{Eqn:Dilworth}
\tilde{\nu} \bigl(\{ \mathbf{y} \in \cY^{\N} : y_{i+1} = T y_i \mbox{ for } i \geq 0 \} \bigr) = 1.
\end{equation}  
The following technical lemma, which relates $\cJ(S: \tilde{\nu})$ to $\cJ(S: \nu)$, will be used in several proofs. 

\vskip.1in

\begin{lemma} \label{Lemma:GoHeels}
If $\tilde{\nu} \in \cM(\cY^{\N}, \tau)$ is generated by the system $(\cY, T, \nu)$, then $m_{\cX \times \cY}(\cJ(S: \tilde{\nu})) = \cJ(S: \nu)$.
\end{lemma}

\begin{proof}
Let $\tilde{\nu} \in \cM(\cY^{\N})$ be generated by $(T,\nu)$.  
Let $\tilde{\lambda} \in \cM(\cX \times \cY^{\N})$ have marginal distribution $\tilde{\nu}$ on $\cM(\cY^{\N})$ and marginal distribution $\eta$ on $\cX \times \cY$.  
Let $f : \cX \times \cY^{\N} \to \R$ be bounded and measurable, and define an associated bounded, measurable function $f_0: \cX \times \cY \to \R$ by $f_0(x,y) = f(x, (y, Ty, \ldots))$.  
Using (\ref{Eqn:Dilworth}), one may verify that
\begin{equation}
\label{lambda-eta}
\int f \, d\tilde{\lambda} \ = \ \int f_0 \, d\eta,
\end{equation}
and for $\tilde{\lambda}$-almost every $(x,\mathbf{y})$,
\begin{align}
\label{f-and-f0}
f \circ (S \times \tau) (x,\mathbf{y})
& \, = \, 
f_0(S x, y_1) 
\, = \, 
f_0 \circ(S \times T) (x, y_0).
\end{align}

Now suppose that $\tilde{\lambda} \in \cJ(S: \tilde{\nu})$ has marginal distribution $\eta$ on $\cX \times \cY$.  
Let $h : \cX \times \cY \to \R$ be a bounded measurable function, and define $f(x, \mathbf{y}) = h(x,y_0)$.  
It then follows from (\ref{lambda-eta}), (\ref{f-and-f0}), and the invariance of $\tilde{\lambda}$ under $S \times \tau$ that $\int h \circ(S \times T) \, d\eta = \int h \, d\eta$.  
As $h$ was arbitrary, $\eta$ is $S \times T$ invariant. 
Furthermore, it is easy to see that the $\cY$-marginal of $\eta$ is $\nu$.  
Thus, $\eta \in \cJ(S: \nu)$.

To establish the other direction, suppose that $\eta \in \cJ(S: \nu)$, and define $\tilde{\lambda} \in \cM(\cX \times \cY^{\N})$ to be the distribution of $(X, Y_0, TY_0, \ldots)$, where $(X, Y_0) \sim \eta$.  
It is clear that $m_{\cX \times \cY}(\tilde{\lambda}) = \eta$ and that the marginal distribution $\tilde{\nu}$ of $\tilde{\lambda}$ on $\cY^{\N}$ is generated by $(T, \nu)$.  
Moreover, it follows from (\ref{lambda-eta}), (\ref{f-and-f0}), and the invariance of $\eta$ under $S \times T$ that $\tilde{\lambda}$ is invariant under $S \times \tau$.
Thus, $\tilde{\lambda} \in  \cJ(S: \tilde{\nu})$, as desired.
\end{proof}

\subsection{A genericity lemma}

The following lemma is standard when $\cU$ is compact (\textit{e.g.}, see \cite{DGS}).  One may reduce the more general case
of interest here to the compact case using the regularity of $\mu$, the separability of $C(K)$ for any compact $K$, 
and the pointwise ergodic theorem. As the argument is straightforward, we omit the proof. 

\begin{lemma} \label{Lemma:GenericPoints}
Suppose $\cU$ is a Polish space, equipped with the Borel $\sigma$-algebra, $R : \cU \to \cU$ is measurable, and $\eta$ is a Borel probability measure on $\cU$ that is ergodic and invariant with respect to $R$. Then there exists a measurable set $E \subset \cU$ such that $\eta( E) = 1$ and if $x$ is in $E$, then for each bounded continuous function $f : \cU \to \R$,
\begin{equation*}
\lim_n \frac{1}{n} \sum_{k=0}^{n-1} f \circ R^k (x) = \int f \, d\eta.
\end{equation*}
\end{lemma}

\section{The \Tracking} \label{Sect:TrackingLemma}

This section is devoted to the proof of Theorem \ref{Thm:GeneralConvergence}.  We first establish the finiteness of the optimal loss $L(S : \nu)$. Recall that under the standard assumptions $\ell^* \in L^1(\nu)$ is a measurable upper bound on $\sup_x |\ell(x,y)|$.

\begin{lemma} 
\label{Lemma:CisFinite}
Under the standard assumptions, $L(S: \nu) \in (-\infty, \infty)$.
\end{lemma}
\begin{proof}
As $\cX$ is non-empty and compact and $S$ is continuous, there exists at least one measure $\mu \in \cM(\cX,S)$.  
Thus $\mu \otimes \nu$ is in $\cJ(S: \nu)$, and in particular, $\cJ(S: \nu)$ is non-empty. 
By assumption, $| \int \ell \, d\lambda | \leq \int \ell^* \, d\nu < \infty$, 
for each $\lambda \in \cJ(S: \nu)$, and as this bound is independent of $\lambda$, the lemma follows.
\end{proof}

The proof of Theorem \ref{Thm:GeneralConvergence} relies on Kingman's subadditive ergodic theorem \cite{Kingman1968,Kingman1973,Kingman1976} and a weak compactness argument.  
We first establish the result when $T$ is  continuous and then deduce the general result from this special case. 

\vskip.1in

\begin{PfofTracking}

We begin by establishing that, for $\nu$-almost every $y$, 
\begin{equation} \label{Eqn:Plato}
\lim_n  \inf_{x \in \cX} \frac{1}{n} \sum_{k=0}^{n-1} \ell(S^k x, \, T^ky) = \sup_n \, \frac{1}{n} \int \biggl(  \inf_{x \in \cX}   \sum_{k=0}^{n-1} \ell\bigl(S^kx, T^ky \bigr) \biggr) \, d\nu(y).
\end{equation} 
For each $n \in \N$ and $y \in \cY$, define
\begin{equation*}
G_n(y) = \inf_{x \in \cX} \sum_{i=0}^{n-1} \ell(S^i x , \, T^i y).
\end{equation*}
Note that the sequence $(G_n)_{n \geq 1}$ is super-additive in the sense that
\begin{align*}
G_{m+n}(y) 
& \, \geq \, \inf_{x \in \cX} \sum_{i=0}^{m-1} \ell(S^i x, \, T^i y) + \inf_{x \in \cX} \sum_{i=m}^{m+n-1} \ell(S^i x , \, T^i y)  \\[.1in]
& \, \geq \, G_{m}(y) + G_{n}(T^m y).
\end{align*}
By Kingman's subadditive ergodic theorem applied to $(-G_n)_{n \geq 1}$,
there exists $\gamma \in (-\infty,\infty]$ such that for $\nu$-almost every $y$,
\begin{equation} 
\label{Eqn:Summertime}
\lim_n \frac{G_n(y)}{n} \ = \ \gamma \ = \ \sup_n \frac{1}{n} \int G_n \, d\nu.
\end{equation}
This equation establishes the existence of the limit in (\ref{Tracking1}) and the equality in (\ref{Eqn:Plato}).

\vskip.1in

We now establish that $L(S: \nu) = \gamma$. 
Let $\lambda$ be any element of $\cJ(S: \nu)$.  As $\lambda$ is invariant under $S \times T$, for each $n \geq 1$,
\begin{align} \label{Eqn:Solstice}
\int  \ell \, d\lambda 
\ = \
\frac{1}{n} \int \sum_{i=0}^{n-1} \ell(S^i x,T^i y) \, d\lambda
\ \geq \
\frac{1}{n} \int G_n \, d\nu.
\end{align}
It then follows from (\ref{Eqn:Summertime}) that
$\int \ell \, d\lambda \geq \gamma$.
As $\lambda \in \cJ(S: \nu)$ was arbitrary, we conclude that $\gamma \leq \inf_{\lambda} \int \ell \, d\lambda = L(S: \nu)$, 
where the infimum is taken over $\lambda$ in $\cJ(S: \nu)$.

\vskip.1in


To complete the proof, we establish the existence of a joining $\lambda \in \cJ(S: \nu)$ such that 
$\int \ell \, d\lambda \leq \gamma$. To do this, we construct a suitable sequence of empirical 
measures on $\cX \times \cY$ and then use a weak compactness argument to identify a limit
$\lambda$ with the desired properties.  Assume for the moment that $T$ is continuous. 

For each $m \geq 1$, let $K_{m} \subset \cY$ be a compact set such that 
$\nu(K_{m}) > 1-\frac{1}{m}$. 
Using the arguments above, Lemma \ref{Lemma:GenericPoints}, and the ergodic theorem, one may identify
a measurable set $E \subseteq \cY$ such that $\nu(E) = 1$ and for every $y \in E$, Equation
(\ref{Eqn:Summertime}) and each of the following relations holds as $n$ tends to infinity: 
\begin{align}
& \nu_n : = \frac{1}{n} \sum_{k=0}^{n-1} \delta_{T^k y} \mbox{ converges weakly to } \nu ; \label{nun_def} \\[.05in]
& \nu_n(K_{m}) \to \nu(K_{m}) 
\ \mbox{ for each $m \geq 1$} ; \\[.1in]
& \int_{\ell^* > m} \ell^* \, d\nu_n \to \int_{\ell^* > m} \ell^* \, d\nu
\ \mbox{ for each $m \geq 1$.} 
\end{align}
Elements of the set $E$ will be referred to as $\nu$-generic points.

Let $y$ be a $\nu$-generic point in $\cY$.  By (\ref{Eqn:Summertime}), there exists a sequence $(x_n)_{n \geq 1}$ in $\cX$ such that
$n^{-1} \sum_{k=0}^{n-1} c(S^k x_n, \, T^k y) \to \gamma$. 
For each $n \geq 1$, define the discrete measure
\begin{equation*}
\lambda_n = \frac{1}{n} \sum_{k=0}^{n-1} \delta_{(S^k x_n, \, T^k y)}
\end{equation*}
on $\cX \times \cY$.
Note that $\lim_n \int \ell \, d\lambda_n = \gamma$.
We claim that the family $\{\lambda_n : n \in \N\}$ is tight.  
To this end, let $\delta > 0$ be given and  choose $N > 1/\delta$. 
By definition, $K_N$ is compact and $\nu(K_N) > 1-\delta$. 
As $y \in E$, for all $n$ sufficiently large,
$\lambda_n(\cX \times K_N) = \nu_n(K_N) > 1 - \delta$. 
As $\delta > 0$ was arbitrary and $\cX$ is compact, the claim follows.


Let $\lambda$ be a weak limit of the family $\{ \lambda_n : n \in \N \}$.  
We claim that $\lambda$ is in $\cJ(S: \nu)$ and that $\lambda$ achieves the infimum in the definition of $L(S : \nu)$ (see (\ref{Tracking1})).  
By passing to a subsequence if necessary, assume that $\lambda_n \Rightarrow \lambda$.  
Let $f$ be in $C_b(\cX \times \cY)$.  
Under the assumption that $S$ and $T$ are continuous, the composition $
f \circ (S \times T)$ is in $C_b(\cX \times \cY)$, and weak convergence then implies that 
\begin{align*}
 \int f \circ (S \times T) \, d\lambda & = \lim_n \int f \circ (S \times T) \, d \lambda_n 
 = \lim_n \frac{1}{n} \sum_{i=1}^{n} f ( S^{i} x_n, T^{i} y) \\
 & = \lim_n \frac{1}{n} \sum_{i=0}^{n-1} f ( S^i x_n, T^i y) 
  = \lim_n  \int f  \, d \lambda_n 
 = \int f \, d\lambda.
\end{align*}
As $f \in C_b(\cX \times \cY)$ was arbitrary, it follows that $\lambda$ is invariant under $S \times T$.  In particular,
$\lambda \in \cM( \cX \times \cY, S \times T)$ and therefore $\pf{\lambda}{\cX}$ is in $\cM(\cX,S)$. 
Furthermore, as $y \in E$, $\pf{\lambda_n}{\cY} = \nu_n$ converges weakly to $\nu$, and therefore
$\pf{\lambda}{\cY} = \nu$.  
Thus, $\lambda$ is in  $\cJ(S: \nu)$.

\vspace{2mm}

In light of the fact that $\lim_n \int \ell \, d\lambda_n = \gamma$, it suffices to show that
\begin{equation} \label{Eqn:Charlotte}
\lim_n \int \ell \, d\lambda_n \geq \int \ell \, d\lambda.
\end{equation}
If the loss $\ell$ were bounded from below, this inequality would follow from the Portmanteau Theorem 
for weak convergence, since we have assumed that it is lower semicontinuous.
For unbounded losses, we appeal to a truncation argument.  
Though the details are somewhat routine, we include them here for completeness.
For $m \in \N$, define the truncated loss
\begin{equation*}
 \ell_{m}(x,y) = \left \{ \begin{array}{ll}
                      \ell(x,y), & \text{ if } |\ell(x,y)| \leq m \\
                       m,  & \text{ if } \ell(x,y) \geq m \\
                       -m, & \text{ if } \ell(x,y) \leq -m.
                     \end{array}
            \right. 
\end{equation*}
Note that $|\ell_{m}| \leq |\ell|$ and that $\ell_{m} \to \ell$ as $m$ tends to infinity.  
The integrability of $\ell$ with respect to $\lambda$ follows from that of $\ell^*$ with respect to $\nu$, and 
the dominated convergence theorem then ensures that $\int \ell_{m} \, d\lambda \to \int \ell \, d\lambda$.  
Moreover, with $\nu_n$ defined as in (\ref{nun_def}), it follows from the choice of $y$ that
\begin{align*} 
\limsup_n
\int |\ell - \ell_{m}| \, d\lambda_n 
\ \leq \
\limsup_n
\int_{\ell^* > m} \ell^* \, d\nu_n 
\ = \
\int_{\ell^* > m} \ell^* \, d\nu .
\end{align*}

In order to establish (\ref{Eqn:Charlotte}), let $\epsilon>0$ be fixed.  By virtue of the results in the previous
paragraph, there exist integers $m$ and $n_1$ sufficiently large that for each $n \geq n_1$
\[
\biggl| \int \ell_{m} \, d\lambda - \int \ell \, d\lambda \, \biggr| < \epsilon/3 
\ \ \mbox{ and } \ 
\int |\ell - \ell_{m}| \, d\lambda_n  < \epsilon/3. 
\]
Moreover, as $\lambda_n \Rightarrow \lambda$ and $\ell_{m}$ is lower semi-continuous 
and bounded, 
there exists $n_2 \geq n_1$ such that for each $n \geq n_2$,
\begin{equation*} 
\int \ell_{m} \, d\lambda - \int \ell_{m} \, d\lambda_n  < \epsilon/3.
\end{equation*}
Combining the inequalities above, a straightforward bound shows that
\begin{equation*}
 \int \ell \, d\lambda - \int \ell \, d\lambda_n   < \epsilon
\end{equation*}
for $n > n_2$.  As $\epsilon>0$ was arbitrary, the inequality (\ref{Eqn:Charlotte}) 
is established, and we conclude that $\gamma \geq L(S: \nu)$.

Suppose now that the transformation $T$ is Borel measurable but not continuous. 
Let $\tilde{\nu}$ be the process measure on $\cY^{\N}$ generated by $(\cY,T,\nu)$ (see Section \ref{Sect:Process}), and let $\tilde{\ell} : \cX \times \cY^{\N} \to \R$ be defined by 
$\tilde{\ell}(x, \mathbf{y}) = \ell(x,y_0)$.  
Note that $\tilde{\ell}$ is lower semicontinuous and that $\sup_{x} |\tilde{\ell}(x, \mathbf{y})|$ is bounded above by a $\tilde{\nu}$-integrable function.
As the left-shift $\tau : \cY^{\N} \to \cY^{\N}$ is continuous, we may apply the arguments above to the systems $(\cX,S)$ and $(\cY^{\N},\tau,\tilde{\nu})$ with loss $\tilde{\ell}$.  
Equation (\ref{Eqn:Summertime}) and inequality (\ref{Eqn:Solstice}) are the same for the original and shift systems.
As for the inequality $\gamma \geq L(S: \nu)$, the arguments above show that there is a joining $\tilde{\lambda} \in \cJ(S: \tilde{\nu})$ such that 
$\gamma = \int \tilde{\ell} \, d\tilde{\lambda} = \int \ell \, d\pf{\tilde{\lambda}}{\cX \times \cY}$. 
By Lemma \ref{Lemma:GoHeels}, $\lambda = \pf{\tilde{\lambda}}{\cX \times \cY}$ is in $\cJ(S: \nu)$.  
This establishes (\ref{Tracking1}) and the existence of a joining $\lambda$ in $\cJ(S: \nu)$ that achieves the infimum in the definition of $L(S: \nu)$. 
\end{PfofTracking}

\section{General results for inference} \label{Sect:GeneralResults}

The present section is devoted to the proof of Theorem \ref{Thm:Inference}.
We begin with several preliminary lemmas, the first of which establishes the existence of optimal tracking schemes.
To ease the notation in the proof of Theorem \ref{Thm:Inference} below, we replace the functions $x_n$ in Definition \ref{Defn:PhiEstimation} with functions $f_n$, as in the following lemma.
 
\begin{lemma}
\label{Lemma:MeasurabilityOptimizers}
If the systems $(\cX,S)$ and $(\cY, T, \nu)$ and the loss 
$\ell$ satisfy the standard assumptions of Section \ref{Sect:Tracking_Intro}, then
there exists a measurable sequence of functions $f_n : \cY^n \to \cX$ satisfying (\ref{Eqn:MinimizingScheme}).
\end{lemma}

\begin{proof}
For each $n \geq 1$, define $\ell_n : \cX \times \cY^n \to \R$ by
$\ell_n(x, y_0^{n-1}) = \sum_{k=0}^{n-1} \ell(S^k x, y_k)$.
Then it is easy to see that $\ell_n = s_n \circ \psi_n$, where 
\[
\psi_n(x, y_0^{n-1}) = ((x, \ldots, S^{n-1}x), y_0^{n-1})
\]
and $s_n(x_0^{n-1}, y_0^{n-1}) = \sum_{k=0}^{n-1} \ell(x_k, y_k)$.
Our assumptions on $S$ and $\ell$ ensure that $\psi_n$ is continuous and that 
$s_n$ is lower semicontinuous, and therefore $\ell_n$ is lower semicontinuous.  
It follows from \cite[Proposition 7.33, p. 153]{Bertsekas1996}
that there exists a Borel measurable function $f_n : \cY^n \to \cX$ such that
\begin{equation*}
\ell_n( f_n(y_0^{n-1}), y_0^{n-1}) = \inf_{ x \in \cX} \ell_n( x, y_0^{n-1}).
\end{equation*}
The definition of $\ell_n$ ensures that $(f_n)_{n \geq 1}$ satisfies the conclusions of the lemma.
\end{proof}

\vskip.1in

\begin{lemma} 
\label{Lemma:ThetaVariational}
Under the standard assumptions,
\begin{equation} 
\label{Eqn:Winter}
  L(S: \nu) = \inf_{\theta \in \Theta} L(S_{\theta}: \nu)
\end{equation}
where $S_{\theta}$ is the restriction of $S$ to $\cX_{\theta} = \varphi^{-1}\{\theta\}$. Furthermore, the infimum is attained.
\end{lemma}

\begin{proof}
Since $(\cX_{\theta}, S_{\theta})$ is a subsystem of $(\cX,S)$, it is immediate that
$L(S: \nu) \leq L(S_{\theta}: \nu)$ for each $\theta \in \Theta$. Thus, it suffices to 
show that $L(S_{\theta}: \nu) \leq L(S: \nu)$ for some $\theta \in \Theta$. 
By Remark \ref{Rmk:ErgJoining}, there exists an 
ergodic joining $\lambda$ in $\cJ_{min}(S: \nu)$.  Define $h: \cX \times \cY \to \Theta$ by
$h = \varphi \circ \proj_{\cX}$ and let $\eta = \lambda \circ h^{-1}$ be the push-forward measure
of $\lambda$ on $\Theta$.  
The assumption that $\varphi \circ S = \varphi$ ensures that $\eta$ is ergodic 
with respect to the identity transformation on $\Theta$, and therefore  
$\eta$ is necessarily a point mass concentrated at some parameter $\theta \in \Theta$.
In particular, $\lambda( \cX_{\theta} \times \cY) = 1$, so that
$\lambda \in \cJ(S_{\theta} : \nu)$.  
Thus $L(S_{\theta} : \nu) \leq \int \ell \, d\lambda = L(S: \nu)$, and the result follows.
\end{proof}

\begin{lemma} \label{Lemma:ThetaMin}
The set $\Theta_{min} \subseteq \Theta$ 
is non-empty and compact.
\end{lemma}

\begin{proof}
As the infimum in (\ref{Eqn:Winter}) is achieved, $\Theta_{min}$ is non-empty.
Since $\Theta$ is compact by assumption, it suffices to show that $\Theta_{min}$ is closed. 
Let $(\theta_n)_{n \geq 1}$ be a sequence in $\Theta_{min}$ that converges to a parameter $\theta \in \Theta$.  
It follows from (\ref{Eqn:EquivDefThetaMin}) that
for each $n$, there is a joining $\lambda_n \in \cJ_{min}(S: \nu)$ such that $\lambda_n( \varphi^{-1} \{\theta_n\} \times \cY) = 1$.  
As $\cJ_{min}(S: \nu)$ is compact, 
the sequence $(\lambda_n)_{n \geq 1}$ has a convergent subsequence.   Passing to a subsequence 
if necessary, suppose that $\lambda_n$ converges to $\lambda \in \cJ_{min}(S: \nu)$.

Define $h : \cX \times \cY \to \Theta$ by $h = \varphi \circ \proj_{\cX}$, and consider 
the associated push-forward measures $\eta_n = \lambda_n \circ h^{-1}$ and 
$\eta = \lambda \circ h^{-1}$ on $\Theta$.  
Note that $\eta_n \Rightarrow \eta$, as $\lambda_n \Rightarrow \lambda$ and $h$ is continuous.  
Our choice of $(\theta_n)$ ensures that 
$\eta_n = \delta_{\theta_n} \Rightarrow \delta_{\theta}$, and as weak limits are
unique, $\eta =  \delta_{\theta}$.   Thus $\lambda(\varphi^{-1}\{\theta\} \times \cY) = 1$, and as
$\lambda \in \cJ_{min}(S: \nu)$ we conclude that $\theta$ is an element of $\Theta_{min}$.
\end{proof}

In the following proof, we make use of the ergodic decomposition of an invariant measure and a related lemma, details of which may be found in Appendix \ref{Sect:SetOfOptimalJoinings}. Note that for the sake of notation in the proof, we denote the sequence of estimators by $f_n : \mathcal{Y}^n \to \mathcal{X}$ instead of by $x_n : \mathcal{Y}^n \to \mathcal{X}$.
\vskip.1in

\begin{PfofInference}

Let $(\theta_n)_{n \geq 1}$ be a sequence of estimators of the form $\theta_n = \varphi \circ f_n$,
where $\{f_n\}$ is a sequence of measurable functions satisfying (\ref{Eqn:MinimizingScheme}), which exists by Lemma \ref{Lemma:MeasurabilityOptimizers}.
As in the proof of Theorem \ref{Thm:GeneralConvergence}, fix a set $E \subset \cY$ 
of $\nu$-generic points having full measure.
Let $y \in E$, and then let $\hat{\theta}_n = \theta_n(y,\dots,T^{n-1}y)$.
 For $n \geq 1$, define 
the state $x_n = f_n(y,\ldots,T^{n-1}y)$ and the associated empirical measure 
$\lambda_n = n^{-1} \sum_{k=0}^{n-1} \delta_{(S^k x_n, T^k y)}$ on $\cX \times \cY$.  
By arguments identical to those in the proof of Theorem \ref{Thm:GeneralConvergence}, one may show that
$(\lambda_n)_n$ is tight and that all of its weak limit points are in $\cJ_{min}(S: \nu)$.

Let $\mathcal{O} \subseteq \Theta$ be an open neighborhood of $\Theta_{min}$. 
Define the function $\psi : \cM( \cX \times \cY ) \to [0,1]$ by 
$\psi(\lambda) = \lambda( \varphi^{-1}(\mathcal{O}) \times \cY)$, and let $V = \psi^{-1} (\frac{1}{2},1]$.
As $\varphi^{-1}(\mathcal{O}) \times \cY$ is open, $\psi$ is lower semi-continuous, and therefore $V$ is open in $\cM( \cX \times \cY)$.
We claim that $\cJ_{min}(S: \nu) \subset V$.  
To see this, first consider an ergodic optimal joining $\eta \in \cJ_{min}(S: \nu)$. 
Define the factor map $h : \cX \times \cY \to \Theta$ by $h = \varphi \circ \proj_{\cX}$.
Since $\eta$ is ergodic, the push-forward measure $\eta \circ h^{-1}$ must be ergodic with respect to the identity map on $\Theta$, as in the proof of Lemma \ref{Lemma:ThetaVariational}, and hence $\eta \circ h^{-1} = \delta_{\theta}$ for some $\theta \in \Theta$.
Thus, for every ergodic $\eta$ there is a parameter $\theta \in \Theta_{min}$ such that 
$\eta(\varphi^{-1} \{\theta\} \times \cY) = 1$, and therefore $\eta(\varphi^{-1} (\Theta_{min}) \times \cY) = 1$.
Now consider an arbitrary joining $\lambda \in \cJ_{min}(S: \nu)$ with ergodic decomposition $\lambda = \int \eta \, d\xi$. 
By Lemma \ref{Lemma:JminErgDecomp}, $\xi$-almost every measure $\eta$ is an ergodic element of $\cJ_{min}(S: \nu)$.
Thus, the above argument gives that for $\xi$-almost every $\eta$, we have $\eta(\varphi^{-1} (\Theta_{min}) \times \cY) = 1$.
It follows that $\psi(\lambda) =1$, and we conclude that $\lambda$ is in $V$.


Since all the limit points of the family $\{\lambda_n\}_{n \geq 1}$ are in the open set $V$, 
there exists an integer $n_1$ such that $\lambda_n \in V$ for all $n \geq n_1$.  
By construction, $\lambda_n( \varphi^{-1} (\mathcal{O}) \times \cY) = \delta_{\hat{\theta}_n}( \mathcal{O} )$.
Thus for each $n \geq n_1$, $\delta_{\hat{\theta}_n}(\mathcal{O}) = \psi( \lambda_n ) > 1/2$, 
which implies that $\hat{\theta}_n$ is in $\mathcal{O}$ as desired.

Let us now show that any parameter in $\Theta_{min}$ is the limit of an optimal $\varphi$-estimation scheme. 
Let $\theta_0$ be any element of $\Theta_{min}$. 
We will show that there exists a $\varphi$-optimal estimation scheme that converges almost surely to $\theta_0$.
By Lemma \ref{Lemma:ThetaVariational}, $L(S : \nu) = L(S_{\theta_0} : \nu)$, and it follows from Remark \ref{Rmk:ErgJoining} that there exists an ergodic joining $\lambda \in \cJ_{min}(S_{\theta_0} : \nu)$.  By Birkhoff's ergodic theorem, there exists a set $E \subset \cX_{\theta} \times \cY$ of $\lambda$ measure
one such that if $(x,y) \in E$ then
\begin{equation} \label{Eqn:Leibniz}
\lim_n \frac{1}{n} \sum_{k=0}^{n-1} \ell(S^kx,T^ky) = \int \ell \, d\lambda,
\end{equation}
which is equal to $L(S: \nu)$ by construction. 
By the regularity of $\lambda$, $E$ contains a $\sigma$-compact set $F$ of the same measure.  
The measurable selection theorem of Brown and Purves \cite[Theorem 1]{Brown1973} then implies 
that there is a measurable function $f : \cY \to \cX_{\theta_0}$ such that $(f(y),y) \in F$ for 
$\nu$-almost every $y \in \cY$. Let $f_n : \cY^n \to \cX$ be given by $f_n(y_0^{n-1}) = f(y_0)$. 
By (\ref{Eqn:Leibniz}), the sequence $\{f_n\}_{n \geq 1}$ satisfies (\ref{Eqn:MinimizingScheme}), 
and therefore the constant estimator $\theta_0 = \varphi \circ f_n$ is an optimal 
$\varphi$-estimation scheme. In particular, there exists an optimal  $\varphi$-estimation scheme that 
converges to $\theta_0$, as desired.
\end{PfofInference}

\section{Quantized observations} \label{Sect:DiscreteObservations}

This section is devoted to the proofs of our results concerning the estimation of transformations from quantized observations.  
We refer to the objects defined in Section \ref{Sect:QuantizedIntro} throughout this section.

Recall that the state space $\cX$ of the reference system is defined to be the closure of the set $\{(\theta, \lab(\theta,u)) : \theta \in \Theta, \, u \in \mathcal{U} \}$ inside of $\Theta \times \{0,1\}^{\N}$. Hence $\cX$ is compact and metrizable, and it is easy to see that $\cX$ is invariant under the 
map $\Id \times \tau$, where $\tau$ is the left-shift map on $\{0,1\}^{\N}$. 
Further recall that the family of transformations in $\cR$ and the topology on $\Theta$ enter indirectly through the definition of $\cX$, which is used to define the target set $\Theta_1$.

First we establish some additional notation and another interpretation of $h(\cR)$, as the topological entropy of a dynamical system.
Let $\cL$ be defined by
\begin{equation*}
\cL \,  = \mbox{closure of } \bigl\{ (\pi(R_{\theta}^k u))_{k \geq 0}  : \theta \in \Theta, \, u \in \cU \bigr\} \mbox{ in $\{0,1\}^\N$} .
\end{equation*} 
Since $\cL$ is a closed subset of $\{0,1\}^\N$, it is compact, and it is easy to see that $\cL$ is invariant 
under the left shift $\tau$.  Thus the pair $(\cL, \tau)$ is a topological dynamical system.  
Now we may observe that the quantity $h(\cR)$ defined in Section \ref{Sect:QuantizedIntro} is actually the topological entropy of the system $(\cL,\tau)$. See  \cite{Walters2000} for an introduction to topological entropy for dynamical systems.

In what follows, we find it useful to state our results in terms of the $\dbar$-distance introduced by Ornstein \cite{Ornstein1970,Ornstein1973,Ornstein1974} 
in the context of the isomorphism theory for Bernoulli processes.
For a fixed shift-invariant measure $\nu$ on $\{0,1\}^{\N}$ and each $\theta \in \Theta$, define
\begin{equation*}
\dbar(\theta : \nu) = \inf_{\mu} \dbar(\mu, \nu),
\end{equation*}
where the infimum is taken over all $\tau$-invariant probability measures $\mu$ such that $\mu(\cX_{\theta}) = 1$.
An application of Theorem \ref{Thm:Inference} yields the following result.

\begin{thm} 
\label{Thm:DiscreteGeneral}
Let $(B_i)_{i \geq 0}$ be a stationary ergodic $\{0,1\}$-valued process having distribution 
$\nu$ on $\{0,1\}^{\N}$.  
If the sequence of estimators $(\theta_n)_n$ satisfies (\ref{Eqn:DiscreteEstimator}), 
then $\hat{\theta}_n(B_0,\dots,B_{n-1})$ converges almost surely to the set
\begin{equation*}
\Theta_0 = \argmin_{\theta \in \Theta} \dbar(\theta : \nu). 
\end{equation*}
\end{thm}

\begin{proof}

By definition, $\Theta_{min} = \argmin_{\theta} L(S_{\theta}: \nu)$.  Using the correspondences in 
Table \ref{Table:DiscreteChoices}, it is easy to see that $S_\theta$ is the restriction of $\Id \times \tau$
to the set $\{\theta\} \times \cX_{\theta}$, and therefore
\[
L(S_{\theta}: \nu) \, = \, \inf_{\lambda \in J(S_{\theta}: \nu)} \int \ell \, d\lambda 
\, = \, \inf_{\mu \in \cM(\cX_{\theta}, \tau)} \dbar(\mu, \nu)
\, = \, \dbar(\theta: \nu). 
\]
Thus $\Theta_{min} = \argmin_{\theta} \dbar(\theta: \nu) = \Theta_0$, as desired.
\end{proof}

By associating parameters $\theta$ with the set of $\tau$-invariant measures 
on $\cX_{\theta}$, we may view the set $\Theta_0$ as the projection of the 
observation process onto the parameter set $\Theta$ under the $\dbar$-metric.
The following lemma, which characterizes when $\dbar(\theta: \eta)$ equals zero, will be used to prove our consistency results.

\begin{lemma} 
\label{Lemma:Dbar}
If $\nu$ is a shift-invariant Borel probability measure on $\{0,1\}^{\N}$, 
then $\dbar(\theta : \nu) = 0$ if and only if $\nu(\cX_{\theta}) = 1$.
\end{lemma}

\begin{proof}
If $\nu(\cX_{\theta}) = 1$, then clearly $0 \leq \dbar(\theta: \nu) \leq \dbar(\nu, \nu) = 0$.
Now suppose that $\dbar(\theta: \eta) = 0$. By Theorem \ref{Thm:GeneralConvergence}, the infimum defining 
$\dbar(\theta: \eta) = L(S_{\theta}: \eta)$ is achieved, and
it follows that there is a measure $\mu \in \cM(\cX_{\theta}, \tau)$ such
that $\dbar(\mu, \nu) = 0$.  As $\dbar$ is a metric, $\mu = \nu$ and therefore 
$\nu(\cX_{\theta}) = 1$.
\end{proof}

\vskip.1in

\vskip.1in

\begin{PfofQuantConsist_I}

By Theorem \ref{Thm:DiscreteGeneral}, the estimates $\hat{\theta}_n(Y_0,\dots,Y_{n-1})$ converge almost surely to the set
$\Theta_0 = \argmin_{\theta \in \Theta} \dbar( \theta : \nu_0)$. Thus it suffices to show that 
$\Theta_0 = \Theta_1$. 
As the measure $\nu_0$ of the observation process is supported on $\cX_{\theta_0}$, we have $\dbar(\theta_0 : \nu_0) = 0$.
Then by Lemma \ref{Lemma:Dbar}, we see that 
$\Theta_0 = \{\theta : \dbar(\theta: \nu_0) = 0\}$ is identical to $\Theta_1 = \{ \theta : \nu_0(\cX_{\theta}) = 1 \}$. 
The converse part of Theorem \ref{Thm:QuantConsist} is a direct consequence of the converse 
part of Theorem \ref{Thm:GeneralConvergence}.
\end{PfofQuantConsist_I}

\vspace{2mm}

Before turning to the proof of Part (2) of Theorem \ref{Thm:QuantConsist}, we require some additional definitions.


\begin{rmk} \label{Rmk:RelIndJoining}
In the proof of Theorem \ref{Thm:QuantConsist} below, 
we make use of a standard construction, called the relatively independent joining, 
to ``glue together'' two joinings along a common factor (see \cite[Section 3.1]{Rue2006}). 
In more detail, suppose we have two measure-preserving system $(\cU_i, R_i, \eta_i)$ for $i = 1, 2$.
Further suppose that these systems have a common factor $(\cU, R, \eta)$, meaning that there exist measurable maps 
$\psi_i : \cU_i \to \cU$ such that $\eta = \eta_i \circ \psi_i^{-1}$ and $\psi_i \circ R_i = R \circ \psi_i$  for $i = 1,2$.  
Now let $\eta_i = \int \eta_{i,u} \, d\eta(u)$ be the disintegration of $\eta_i$ over $\eta$, and define the measure $\lambda$ on $\cU_1 \times \cU_2$ by
\begin{equation*}
\lambda = \int \eta_{1,u} \otimes \eta_{2,u} \, d\eta(u).
\end{equation*}
Then it is not difficult to check that $\lambda$ is a joining of $(\cU_1, R_1, \eta_1)$ and $(\cU_2, R_2, \eta_2)$ such that 
$
\lambda \bigl( \bigl\{ ( u_1, u_2) \in \cU_1 \times \cU_2 : \psi_1( u_1) = \psi_2( u_2) \bigr\} \bigr) = 1.
$
\end{rmk}

\vskip.1in 
 
\begin{rmk} \label{Rmk:VarPrinciple}
The proof of Theorem \ref{Thm:QuantConsist} also requires some elementary facts concerning 
the entropy of dynamical systems (see \cite{Walters2000} for a thorough treatment of the subject). 
Let $\Sigma \subset \{0,1\}^{\N}$ be a closed, shift-invariant set of label sequences. 
The topological entropy $h_{top}(\Sigma,\tau)$ of the system $(\Sigma, \tau)$ is given in (\ref{Eqn:Htop}). 
If $\eta$ is a shift-invariant Borel probability measure on such a set $\Sigma$, then the measure theoretic 
(Kolmogorov-Sinai) entropy of the system $(\Sigma, \tau, \eta)$ is defined by
\begin{equation*}
h(\eta) = \lim_n - \frac{1}{n+1} \sum_{a_0^n \in \Sigma_n} \eta([a_0^n]) \log \eta([a_0^n]),
\end{equation*}
where $\Sigma_n = \{ a_0^n \in \{0,1\}^{n+1} : \mathbf{a} \in \Sigma \}$ and $[a_0^n]$ denotes the cylinder 
set of sequences $\mathbf{a} \in \{0,1\}^{\N}$ whose first $n+1$ coordinates are $a_0^n$.
The well-known Variational Principle (see \cite{Walters2000}) states that
\begin{equation*}
 h_{top}(\Sigma,\tau) = \sup_{\eta \in \cM(\Sigma,\tau)} h(\eta).
\end{equation*}
Thus if $h_{top}(\Sigma,\tau) = 0$, then $h(\eta) = 0$ for any measure $\eta$ in $\cM(\Sigma,\tau)$.
\end{rmk}

\begin{PfofQuantConsist_II}

By Theorem \ref{Thm:DiscreteGeneral}, any sequence of estimates satisfying (\ref{Eqn:DiscreteEstimator}) 
converges almost surely to the parameter set 
$\Theta_0 = \argmin_{\theta \in \Theta} \dbar(\theta: \nu)$,
where $\nu$ is the measure of the observation process $(Y_i)_{i \geq 0}$, 
which involves errors.  By contrast, we have $\Theta_1 = \{\theta \in \Theta : \nu_0(L_{\theta}) = 1\}$, 
where $\nu_0$ is the measure for the error-free process $( \pi( R_{\theta_0}^i U))_{i \geq 0}$. 
It therefore suffices to show that $\theta$ minimizes $\dbar(\theta : \nu)$ if and only 
if $\nu_0(\cX_{\theta}) = 1$. Additionally, with this relation established, the converse
part of Theorem \ref{Thm:QuantConsist} is a direct consequence of the converse part of Theorem 
\ref{Thm:GeneralConvergence}.

Let $p = \mathbb{P} \big( \epsilon_0 =1 \bigr)$, which is less than $1/2$ by hypothesis.
We claim that for each $\theta$ in $\Theta$,
\begin{equation} 
\label{Eqn:Hornets}
\dbar(\theta: \nu) \ \geq \ p \, + \, \dbar(\theta : \nu_0) \, (1-2p).
\end{equation}
To see this, fix $\theta \in \Theta$.  Then Theorem \ref{Thm:GeneralConvergence} with $(\cX, S) = (\cX_{\theta},\tau)$ ensures that  
there is an optimal joining $\lambda_1$ of $\theta$ and $\nu$.  In detail, there exists $\lambda_1 \in \cM(\cX_{\theta} \times \{0,1\}^{\N}, \tau \times \tau)$ 
such that its first marginal, $\mu$ say, is supported on $\cX_{\theta}$, its second marginal is equal to $\nu$, 
and $\dbar(\theta: \nu) = \lambda_1( \{ (\mathbf{a}, \mathbf{b})  : a_0 \neq b_0 \})$.

Let $\eta$ be the process measure for the noise process $(\epsilon_i)_{i \geq 0}$, 
and let $\lambda_2$ denote the product measure $\nu_0 \otimes \eta$. 
Define the maps $\psi_1 : \{0,1\}^{\N} \times \{0,1\}^{\N} \to \{0,1\}^{\N}$ and $\psi_2 : \{0,1\}^{\N} \times \{0,1\}^{\N} \to \{0,1\}^{\N}$ by $\psi_1(\mathbf{a},\mathbf{b}) = \mathbf{b}$ 
and $\psi_2( \mathbf{v}, \boldsymbol{\epsilon}) = ( v_i \oplus \epsilon_i)_{i \geq 0}$. 
Note that $\nu$ is a common factor of $\lambda_1$ and $\lambda_2$ under the maps 
$\psi_1$ and $\psi_2$, respectively. 
Using the construction discussed in Remark \ref{Rmk:RelIndJoining}, one may construct a 
joining $\tilde{\lambda}$ of $\lambda_1$ and $\lambda_2$ such that
\begin{equation*}
\tilde{\lambda} 
\biggl( \bigl\{ (\mathbf{a},\mathbf{b},\mathbf{v},\boldsymbol{\epsilon}) : b_i = v_i \oplus \epsilon_i
\text{ for all } i \geq 0 \bigr\} \biggr) = 1.
\end{equation*}
Let $(A_i, B_i, V_i, \epsilon_i)_{i \geq 0}$ denote the multi-label process having distribution $\tilde{\lambda}$. 
By construction, the following hold:
\begin{enumerate}

\vskip.05in

\item
$(A_i)_{i \geq 0} \sim \mu$ and is supported on $\cX_{\theta}$; 

\vskip.06in

\item
$(B_i)_{i \geq 0} \sim \nu$; 

\vskip.06in

\item
$(V_i)_{i \geq 0} \sim \nu_0$; 

\vskip.06in

\item
$(\epsilon_i)_{i \geq 0} \sim \eta$ is a copy of the i.i.d.\ noise process; 

\vskip.06in

\item
$(A_i, B_i)_{i \geq 0} \sim \lambda_1$, the joining of $\mu$ and $\nu$; 

\vskip.06in

\item
$B_i = V_i \oplus \epsilon_i$ almost surely.

\vskip.05in

\end{enumerate}
From these properties and elementary arguments
we see that
\begin{align*}
\dbar(\theta: \nu) 
& \ = \ \lambda_1( A_0 \neq B_0) 
\ = \ \tilde{\lambda}( A_0 \neq (V_0 \oplus \epsilon_0) ) \\[.06in]
& \ \geq \ \tilde{\lambda} \Bigl( \{ \epsilon_0 = 1, \, A_0 = V_0 \} 
                                             \cup \{ \epsilon_0 = 0, \, A_0 \neq V_0 \} \Bigr) \\[.06in]
& \ = \ \tilde{\lambda} \bigl( \epsilon_0 = 1, \, A_0 = V_0 \bigr) + \tilde{\lambda} \bigl( \epsilon_0 = 0, \, A_0 \neq V_0  \bigr).
\end{align*}

As the measures $\mu$ and $\nu_0$ are supported on $\cL$, the assumption that $h(\cR) = 0$ 
implies that $h(\mu) = 0$ and $h(\nu_0) = 0$ (see Remark \ref{Rmk:VarPrinciple}). 
Let $\lambda_3$ be the joining of $\mu$ and $\nu_0$ given by the marginal distribution of 
$\tilde{\lambda}$ on $(\bf{a}, \bf{v})$. 
By a standard bound on entropy, $h(\lambda_3) \leq h(\mu) + h(\nu_0) = 0$, and therefore $h(\lambda_3)= 0$. 
It follows from a classical result of Furstenberg \cite[Theorem I.2]{Furstenberg1967} 
that the only joining between the zero-entropy measure $\lambda_3$ and the i.i.d.\ measure $\eta$ 
is the product (independent) joining $\lambda_3 \otimes \eta$.  
Consequently, $(A_0,V_0)$ and $\epsilon_0$ are independent under $\tilde{\lambda}$.  Therefore
\begin{align*}
\tilde{\lambda} \bigl( \epsilon_0 = 1, \, A_0 = V_0 \bigr) 
= 
\tilde{\lambda} \bigl( \epsilon_0 = 1 \bigr) \, \tilde{\lambda} \bigl( A_0 = V_0 \bigr) 
= 
p \, \bigl(1- \tilde{\lambda}( A_0 \neq V_0 ) \bigr),
\end{align*}
and 
\begin{align*}
\tilde{\lambda} \bigl( \epsilon_0 = 0, \, A_0 \neq V_0 \bigr) 
= 
\tilde{\lambda} \bigl( \epsilon_0 = 0 \bigr) \, \tilde{\lambda} \bigl( A_0 \neq V_0 \bigr) 
= 
(1-p) \, \tilde{\lambda}( A_0 \neq V_0 ).
\end{align*}
Combining the previous three displays gives
\begin{align*}
\dbar(\theta: \nu) 
& \ = \ \lambda_1 \bigl( A_0 \neq B_0 \bigr) \\[.06in]
& \ \geq p \ \bigl(1- \tilde{\lambda}( A_0 \neq V_0 ) \bigr) + (1-p) \tilde{\lambda}( A_0 \neq V_0 ) \\[.06in]
& \ = \ p + \tilde{\lambda}( A_0 \neq V_0 ) (1-2p).
\end{align*}
Under the joining $\tilde{\lambda}$, $(A_i)_{i \geq 0}$ is distributed according to 
$\mu$, which is supported on $\cX_{\theta}$, 
and $(V_i)_{i \geq 0}$ is distributed according to $\nu_0$.  Thus
$\tilde{\lambda}( A_0 \neq V_0) \geq \dbar(\theta: \nu_0)$,
and the inequality (\ref{Eqn:Hornets}) follows from the previous display as $p < 1/2$.

It follows from (\ref{Eqn:Hornets}) that $\dbar(\theta: \nu) \geq p$ for all $\theta$.  We now show
that $\dbar(\theta_0 : \nu) = p$, from which it follows that $\Theta_0 = \{ \theta : \dbar(\theta: \nu) = p\}$.
Let $\lambda = \nu_0 \otimes \eta \in \cM(\{0,1\}^{\N} \times \{0,1\}^{\N})$ and let 
$\psi : \{0,1\}^{\N} \times \{0,1\}^{\N} \to \{0,1\}^{\N} \times \{0,1\}^{\N}$ be defined by 
$\psi(\mathbf{a},\boldsymbol{\epsilon}) = (\mathbf{a}, (a_i \oplus \epsilon_i)_{i \geq 0})$.
It is straightforward to show that $\lambda_4 = \lambda \circ \psi^{-1}$ is a joining of 
$\nu_0$ with $\nu$ such that
$\lambda_4 ( \{ (\mathbf{a},\mathbf{b}) : a_0 \neq b_0 \} ) = p$,
and therefore $\dbar(\theta_0: \nu) = p$ as desired.

Let $\theta \in \Theta_0$.  The arguments above show that $\dbar(\theta : \nu) = p$, and then it follows from 
(\ref{Eqn:Hornets}) that $\dbar(\theta: \nu_0) = 0$.  Hence $\nu_0(\cX_{\theta}) = 1$ by Lemma \ref{Lemma:Dbar}, and we conclude that $\theta \in \Theta_1$.  This shows that $\Theta_0 \subseteq \Theta_1$. 
For the reverse inclusion, we note that if $\nu_0(\cX_{\theta}) = 1$, then (\ref{Eqn:Hornets}) and the joining $\lambda_4$ 
can be used to show that $\dbar(\theta: \nu) = p$ (as we did for $\theta_0$), which implies that $\theta$ is in $\Theta_0$.
\end{PfofQuantConsist_II}

\vspace{2mm}

If the partition $\pi$ does not not resolve differences between the generative transformation $R_{\theta_0}$ 
and other transformations on the support of $\nu_0$, then $\Theta_1$ may not be equal to $\{\theta_0\}$. 
The following result provides conditions under which $\Theta_1$ is a singleton.

\vskip.1in

\begin{prop}  \label{Prop:Identifiability}
\label{Prop:UniqueThetaStar_Intro}
Suppose that for all $\theta' \neq \theta_0$ there exists a neighborhood $\mathcal{O}$ of 
$\theta'$ and an integer $N$ depending on $\theta_0$ and $\mathcal{O}$ such that 
for all $u,v \in \cU$ and all $\theta \in \mathcal{O}$,
$\pi( R_{\theta_0}^k u) \neq \pi( R_{\theta}^k v )$ for some $k \in [0,N]$.
Then $\Theta_1 = \{\theta_0\}$.
\end{prop}
\begin{proof}
Suppose the hypotheses of the proposition hold, and let $\theta' \neq \theta_0$.  We will show that
$\nu_0(\cX_{\theta'}) = 0$, and therefore $\theta' \notin \Theta_1$. 
Let $\mathcal{O}$ and $N$ be as in the statement of the proposition.  For $\theta \in \Theta$ 
and $u \in \cU$ define
\begin{equation*}
C^N(\theta,u) = \{ \mathbf{a} \in \{0,1\}^{\N} : a_k = \pi(R_{\theta}^k u) \mbox{ for } 0 \leq k \leq N \},
\end{equation*}
and for $\Theta' \subseteq \Theta$ let $C^N(\Theta')$ be the union of $C^N(\theta,u)$ over
$\theta \in \Theta'$ and $u \in \cU$. 
The cylinder sets $C^N(\theta,u)$ are closed and open, and since $N$ is fixed, there are finitely many of them. 
As $\mathcal{O}$ is an open neighborhood of $\theta'$ it is clear that $\cX_{\theta'} \subset C^N(\mathcal{O})$.
Moreover, as $\nu_0$ is supported on the set of label sequences generated by $\theta_0$ and 
$C^N(\theta_0)$ is closed, $\supp(\nu_0) \subset C^N(\theta_0)$.
The hypotheses of the proposition imply that that $C^N(\mathcal{O})$ and $C^N(\theta_0)$ are disjoint,
and therefore $\cX_{\theta'}$ is disjoint from $\supp(\nu_0)$.
\end{proof}

\vspace{2mm}

\vskip.1in 
 
\begin{prop} \label{Prop:CircleRotIdentifiable}
Under the hypotheses in Section \ref{Ex:CircleRotations}, we have $\Theta_1 = \{\alpha_0\}$.
\end{prop}

\begin{proof}
Recall that $\Theta = [0,\frac{1}{2}]$.  We wish to apply Proposition \ref{Prop:UniqueThetaStar_Intro}. 
Let $\alpha_1 \neq \alpha_2$ and assume without loss of generality that $\alpha_1 < \alpha_2$. 
Fix $0 < \epsilon < (\alpha_2 - \alpha_1) / 2$ and let 
$\mathcal{O} = [\alpha_2-\epsilon, \alpha_2+\epsilon] \cap \Theta$ and 
$N \geq 3 / 2 (\alpha_2 - \alpha_1 -\epsilon)$.
Let $u, v \in [0,1)$ and $\alpha \in \mathcal{O}$. Define
\begin{equation*}
k = \inf \biggl\{ j \geq 0 : | (u + j \alpha) - (v+j \alpha_1)| \geq \frac{1}{2} \biggr\}.
\end{equation*}
Our choice of $N$ ensures that $N(\alpha - \alpha_1) \geq 3/ 2 \geq 1/2 + v - u$.  Therefore
 \begin{equation*}
   u + N \alpha - (v + N \alpha_1) \geq \frac{1}{2},
\end{equation*}
so that $k \leq N$. 
We claim that $\pi(R_{\alpha}^k u) \neq \pi(R_{\alpha_1}^k v)$. 
If $k = 0$ then $\frac{1}{2} \leq |u-v| < 1$, which implies that $\pi(u) \neq \pi(v)$. 
Suppose that $k \geq 1$. Using the definition of $k$ (twice) and the triangle inequality, we see that
\begin{align*}
\frac{1}{2} & \leq | (u + k \alpha) - (v + k \alpha_1)| \\[.06in]
  & \leq |(u + (k-1) \alpha) - (v+ (k-1)\alpha_1)| + |\alpha - \alpha_1| \\
  & < \frac{1}{2} + \frac{1}{2} 
   = 1,
\end{align*}
and therefore $\pi(R_{\alpha}^k u) \neq \pi(R_{\alpha_1}^k v)$. 
Now Proposition \ref{Prop:UniqueThetaStar_Intro} yields the result.
 \end{proof}

\vspace{2mm}

\begin{PfofCircleRotations}

We first show that $h(\cR) = 0$. 
For $u$ in $[0,1)$ and $\alpha$ in $\Theta$, let $\pi^n(u,\alpha)$ denote 
the element $w = w_0 \dots w_{n-1}$ of $\{0,1\}^{n}$ such that 
$R_{\alpha}^k u \in A_{w_k}$ for $k=0,\dots,n-1$.  
Define
\begin{equation*}
   C(n) = \biggl| \biggl\{ w \in \{0,1\}^{n} : \exists u \in [0,1), \exists \alpha \in \Theta, \, w = \pi^n(u,\alpha) \biggr\} \biggr|,
\end{equation*}
and note that
\begin{equation*}
h(\cR) \leq \limsup_n \frac{1}{n} \log C(n).
\end{equation*}
It is known (see \cite{Frid2014}) that $C(n) \leq K n^4$ for some constant $K$, and it then follows
from the previous display that $h(\cR) = 0$.
By Theorem \ref{Thm:QuantConsist}, any estimates satisfying 
(\ref{Eqn:DiscreteEstimator}) converge almost surely to $\Theta_1$, and $\Theta_1 = \{\alpha_0\}$ by 
Proposition \ref{Prop:CircleRotIdentifiable}.
\end{PfofCircleRotations}

\section*{Acknowledgements}

The authors thank Sayan Mukherjee and Karl Petersen for productive conversations.
Both authors acknowledge the support of the National Science Foundation under grant DMS 1613261, and AN also acknowledges support under grant DMS 1310002.


\bibliographystyle{plain}
\bibliography{Optimization_Fundamentals_refs}

\vskip.2in

\appendix
\section{Structure of the set of optimal joinings} 
\label{Sect:SetOfOptimalJoinings}

In this section, we investigate the structure of the set of optimal joinings in Theorem \ref{Thm:GeneralConvergence} and provide the proof of Theorem \ref{Thm:BasicPropertiesOfJmin}.  
Our results rely on a general version of the ergodic decomposition for invariant probability measures. 
The following version, a restatement of \cite[Theorem 2.5]{Sarig2008}, is sufficient for our purposes.

\begin{thm*}[The Ergodic Decomposition]
Suppose that $R : \cU \to \cU$ is a Borel measurable map of a Polish space $\cU$ and that $\lambda \in \cM(\cU,R)$. 
Then there exists a Borel probability measure $\xi$ on $\cM(\cU)$ 
such that 
\renewcommand{\theenumi}{\arabic{enumi}}
\renewcommand{\labelenumi}{(\theenumi)}
\begin{enumerate}
\vskip.07in
\item $\xi\bigl( \{ \eta \mbox{ is invariant and ergodic for $R$}  \} \bigr) = 1$
\vskip.1in
\item If $f \in L^1(\lambda)$, then $f \in L^1(\eta)$ for $\xi$-almost every $\eta$, and
        \begin{equation*}
         \int f \, d\lambda = \int \biggl(\int f \, d\eta \biggr) \, d\xi(\eta).
        \end{equation*}
\end{enumerate}
Whenever (2) holds, we write $\lambda = \int \eta \, d\xi$.
\end{thm*}

\begin{rmk} \label{Rmk:Pushforward}
Suppose $f : \cU \to \cV$ is a Borel measurable map between Polish spaces $\cU$ and $\cV$ 
and that $\lambda \in \cM(\cU)$ satisfies $\lambda = \int \eta \, d\xi$. Then one may readily
check that $\pf{\lambda}{f} = \int \pf{\eta}{f} \, d\xi$.
\end{rmk}

In the remainder of this section we assume that the systems $(\cX,S)$ and $(\cY, T, \nu)$ 
and the loss $\ell$ satisfy the standard assumptions in 
Section \ref{Sect:Tracking_Intro}. 
The following three lemmas will be used to prove Theorem \ref{Thm:BasicPropertiesOfJmin}.

\begin{lemma} \label{Lemma:Botanical} 
If $\lambda \in \cJ(S: \nu)$ has ergodic decomposition $\lambda = \int \eta \, d\xi$,
then $\xi$-almost every $\eta$ is in $\cJ(S: \nu)$. 
\end{lemma}

\begin{proof}
Let $\lambda \in \cJ(S: \nu)$ have ergodic decomposition $\lambda = \int \eta \, d\xi$. 
Then $\xi$-almost every $\eta$ is in $\cM(\cX \times \cY, S \times T)$, and for these measures $\pf{\eta}{\cX}$ is necessarily in $\cM(\cX,S)$. 
It follows from Remark \ref{Rmk:Pushforward} that $\nu = \pf{\lambda}{\cY} = \int \pf{\eta}{\cY} \, d\xi$. 
Since $\nu$ is ergodic, it is an extreme point of the convex set $\cM(\cY,T)$ (see \cite[Proposition 12.4]{Phelps2001}), 
and therefore $\pf{\eta}{\cY} = \nu$ for $\xi$-almost every $\eta$. 
\end{proof}

\begin{lemma} \label{Lemma:JminErgDecomp}
If $\lambda \in \cJ_{min}(S: \nu)$ has ergodic decomposition $\lambda = \int \eta \, d\xi$,
then $\xi$-almost every $\eta$ is ergodic and contained in $\cJ_{min}(S: \nu)$.
\end{lemma}

\begin{proof}
Ergodicity of $\xi$-almost every $\eta$ follows from the definition of the ergodic decomposition.
By assumption, $\ell \in L^1(\lambda)$, and the ergodic decomposition yields 
\begin{equation*} 
\int \ell \, d\lambda = \int \biggl( \int \ell \, d\eta \biggr) \, d\xi(\eta).
\end{equation*}
By Lemma \ref{Lemma:Botanical}, $\xi$-almost every $\eta \in \cJ(S: \nu)$ 
and in this case $\int \ell \, d\lambda \leq \int \ell \, d\eta$,
as $\lambda \in \cJ_{min}(S: \nu)$.  It follows that
$\int \ell \, d\lambda = \int \ell \, d\eta$ for $\xi$-almost every $\eta$.
\end{proof}

\begin{lemma} \label{Lemma:PhigIsUSC}
The functional $\phi_{\ell} : \cJ( S: \nu) \to \R$ defined by
$\phi_{\ell}( \lambda ) = \int \ell \, d\lambda$ is lower semi-continuous.
\end{lemma}

If the loss $\ell$ is bounded from below, then the conclusion of this lemma follows immediately from the Portmanteau Theorem for weak convergence. 
In the general case, the result may be established by a truncation argument; as the argument is very similar to that in the proof of Theorem \ref{Thm:GeneralConvergence}, we omit the details.  


\begin{prop} \label{Prop:Marcus}
The set of measures $\cJ(S: \nu)$ is compact in the weak topology.
\end{prop}

\begin{proof}
We address the general case of measurable $T$.  If $T$ is continuous, then the proof may be simplified.
Let $\tilde{\nu} \in \cM(\cY^{\N})$ be the process measure generated by $(T,\nu)$ (see Section \ref{Sect:Process}),
and let $\cJ(S: \tilde{\nu})$ be the associated set of joinings.

We claim that $\cJ(S: \tilde{\nu})$ is compact in $\cM(\cX \times \cY^{\N})$.  
To see this, note that the direct product of $S$ with the left-shift $\tau$ on $\cY^{\N}$ is continuous, since each of these maps is continuous. 
It then follows that the induced push-forward map $m_{S \times \tau}$ from $\cM(\cX \times \cY^{\N})$ to itself is continuous in the weak topology.
The set of fixed points of any continuous map is closed. 
Hence the set $C_1 = \cM(\cX \times \cY^{\N}, S \times \tau)$ of fixed points of $m_{S \times \tau}$ is closed.  
Additionally, the continuity of the projection $m_{\cY^{\N}} : \cM(\cX \times \cY^{\N}) \to \cM(\cY^{\N})$ ensures that $C_2  = m_{\cY^{\N}}^{-1}\{\tilde{\nu}\}$ is also closed, and therefore our set of interest $\cJ(S: \tilde{\nu}) = C_1 \cap C_2$ is closed.  
Now let $\epsilon >0$. Since $\{\tilde{\nu}\}$ is tight, there exists a compact set $K \subset \cY^{\N}$ such that $\tilde{\nu}(K) > 1-\epsilon$.
By assumption $\cX$ is compact, and therefore $\cX \times K$ is compact. 
Also, for any $\lambda \in \cJ(S: \tilde{\nu})$, we have $\lambda(\cX \times K) = \tilde{\nu}(K) > 1-\epsilon$. 
Since $\epsilon>0$ was arbitrary, we see that $\cJ(S: \tilde{\nu})$ is tight. 
Then, since it is both tight and closed, we may conclude that it is compact.

As the projection $(x,\mathbf{y}) \mapsto (x,y_0)$ is continuous, the induced 
push-forward map $m_{\cX \times \cY} : \cM(\cX \times \cY^{\N}) \to \cM(\cX \times \cY)$ 
is continuous.
By Lemma \ref{Lemma:GoHeels}, $m_{\cX \times \cY}(\cJ(S: \tilde{\nu})) = \cJ(S: \nu)$, and 
therefore $\cJ(S: \nu)$ is compact, as it is the image of a compact set under a continuous map. 
\end{proof}

\vspace{2mm}

\begin{PfofOptJoinings}

Let  $\phi_{\ell} : \cJ( S: \nu) \to \R$ be the functional defined by
$\phi_{\ell}( \lambda ) = \int \ell \, d\lambda$.
The convexity of $\cJ_{min}(S: \nu)$ follows from the fact that 
$\phi_{\ell}(t \lambda_1 + (1-t) \lambda_2) = t \phi_{\ell}(\lambda_1) + (1-t) \phi_{\ell}( \lambda_2)$.
As $\cJ(S: \nu)$ is compact in the weak topology (Proposition \ref{Prop:Marcus}) and $\phi_{\ell}$ is lower semicontinuous (Lemma \ref{Lemma:PhigIsUSC}), 
the set $\cJ_{min}(S: \nu)$ where $\phi_{\ell}$ attains its minimum is compact.  

It remains to identify the extreme points of $\cJ_{min}(S: \nu)$.   
Any ergodic measure in the (convex) set $\cM(\cX \times \cY, S \times T)$ is an extreme 
point of this set (see \cite[Proposition 12.4]{Phelps2001}), and it follows that any
ergodic measure $\lambda \in \cJ_{min}(S: \nu)$ is an extreme point of $\cJ_{min}(S: \nu)$. 
Suppose now that $\lambda$ is an extreme point of $\cJ_{min}(S: \nu)$, and 
let $\lambda = \int \eta \, d\xi$ be its ergodic decomposition. 
By Lemma \ref{Lemma:JminErgDecomp}, $\xi$-almost every $\eta$ is in $\cJ_{min}(S: \nu)$, and 
as $\lambda$ is an extreme point of $\cJ_{min}(S: \nu)$, it follows that
$\xi$-almost every $\eta$ equals $\lambda$.  
Then it follows from the ergodic decomposition that  $\lambda$ is ergodic.
\end{PfofOptJoinings}

\end{document}